\documentclass[preprint,3p,12pt]{elsarticle}

\usepackage{amssymb}
\usepackage{amsmath}
\usepackage{amsthm}

\newtheorem{theorem}{Theorem}[section]
\newtheorem{lemma}[theorem]{Lemma}

\newtheorem{assumption}[theorem]{Assumption}
\newtheorem{remark}[theorem]{Remark}
\newtheorem{proposition}[theorem]{Proposition}

\usepackage{lineno}

\usepackage{cleveref}
\usepackage{amssymb}
\usepackage{mathrsfs}

\newcommand{\eq}{:=}

\newcommand{\grad}{\boldsymbol \nabla}
\renewcommand{\div}{\grad \cdot}

\newcommand{\ddiv}{\operatorname{div}}

\newcommand{\jmp}[1]{[\![#1]\!]}

\newcommand{\BH}{\boldsymbol H}

\newcommand{\BL}{\boldsymbol L}

\newcommand{\BV}{\boldsymbol V}

\newcommand{\be}{\boldsymbol e}
\newcommand{\bn}{\boldsymbol n}

\newcommand{\bq}{\boldsymbol q}
\newcommand{\br}{\boldsymbol r}

\newcommand{\bv}{\boldsymbol v}

\newcommand{\CF}{\mathcal F}

\newcommand{\CO}{\mathcal O}
\newcommand{\CP}{\mathcal P}

\newcommand{\CR}{\mathcal R}

\newcommand{\CT}{\mathcal T}

\newcommand{\BCP}{\boldsymbol{\CP}}

\usepackage{caption}
\usepackage[skip=1pt]{subcaption}

\DeclareMathOperator*{\argmin}{arg\,min}

\begin{document}

\begin{frontmatter}



\title{Minimum-residual a posteriori error estimates for HDG discretizations of the Helmholtz equation}


\author[1]{Liliana Camargo} 
\author[2]{Sergio Rojas} 
\author[3,4]{Patrick Vega} 

\affiliation[1]{organization={Departamento de Matemáticas, Universidad de La Serena},
            city={La Serena}, 
            country={Chile}}

\affiliation[2]{organization={School of Mathematics, Monash University},
            city={Melbourne},
            country={Australia}}

\affiliation[3]{organization={Departamento de Matemática y Ciencia de la Computación\\Universidad de Santiago de Chile},
            city={Santiago}, 
            country={Chile}}

\affiliation[4]{organization={Centro de Investigación en Ingeniería Matemática (CI$^2$MA)},
            city={Concepción}, 
            country={Chile}}

\begin{abstract}
We propose and analyze two a posteriori error indicators for hybridizable discontinuous Galerkin (HDG) discretizations of the Helmholtz equation. These indicators are built to minimize the residual associated with a local superconvergent postprocessing scheme for the primal variable, measured in a dual norm of an enlarged discrete test space. The residual minimization is reformulated into equivalent local saddle-point problems, each yielding a superconvergent postprocessed approximation of the primal variable in the asymptotic regime for sufficiently regular exact solutions and a built-in residual representation with minimal computational effort. Both error indicators are based on frequency-dependent postprocessing schemes and verify reliability and efficiency estimates for a frequency-weighted $H^1$-error for the scalar variable and the $L^2$-error for the flux. We illustrate our theoretical findings through ad-hoc numerical experiments.
\end{abstract}




\end{frontmatter}




\tableofcontents

\section{Introduction}
\label{intro}

The Helmholtz equation serves as an essential model for many physical applications involving time-harmonic wave propagation phenomena, such as elastodynamics, electrodynamics, and linear acoustic wave propagation  \cite{farhat_harari_hetmaniuk_2003,hu_song_2020,li_liu_yang_2023,nguyen_peraire_reitich_cockburn_2015}. However, the analysis of numerical schemes based on variational formulations deduced from Helmholtz problems with large wave number remains challenging due to the indefiniteness of the associated sesquilinear forms.

Significant contributions have been made to the error analysis of discretization methods used for the numerical solution of the Helmholtz equation with high wavenumber; among them, we highlight finite element \cite{harari_hughes_1991,harari_hughes_1992,ihlenburg_babuska_1995,thompson_pinsky_1995,zhu_wu_2014}, $hp$-finite element \cite{ihlenburg_babuska_1997,melenk_sauter_2011,zhu_wu_2013}, discontinuous Galerkin \cite{feng2009discontinuous,feng_wu_2011,feng_xing_2013,melenk_parsiania_sauter_2013} and Trefftz methods \cite{Hiptmair2016}. A key challenge in these methods is the pollution effect \cite{babuska1997pollution}, which significantly impacts preasymptotic error estimates. This effect is typically characterized by a pollution term, deduced using data estimates or duality arguments \cite{ihlenburg_babuska_1997}. In \cite{zhu2021preasymptotic} and references therein, it is shown that the pollution effect can be reduced by selecting appropriate penalty parameters in 2D. Moreover, they observed that while their postprocessing procedure improved accuracy, the convergence order of the pollution term remained unchanged.  

Discontinuous Galerkin (DG) methods offer several attractive features in the context of the Helmholtz equation. Notable among these are their ability to accommodate variable-degree approximations across elements, which helps control dispersion effects by compensating for varying element sizes, and their capacity to handle nonconforming meshes, such as those with hanging nodes. However, a significant drawback of DG methods is the increased dimension of the standard DG piecewise polynomial space compared to the associated conforming space in $H^1$. This limitation is particularly troublesome in high-frequency wave propagation problems, where the dimensionality of the approximation space grows faster than $\CO(\omega^d)$ in two or three dimensions ($d=2, 3$). As a result, the space's size scales with the wavenumber $\omega$, posing computational challenges as $\omega$ increases. 

Since many practical problems involve high frequencies, DG methods may appear less useful than conforming FEM as a consequence of the limitation mentioned above. However, the hybridizable discontinuous Galerkin (HDG) method \cite{CoDoGu2008, CoGoLa2009} overcomes this drawback by significantly reducing the size of the associated linear systems. 

HDG methods are a class of DG methods that introduce numerical fluxes defined on the skeleton of the mesh. These flux variables enable the solution within each element to be expressed in terms of local quantities, taking advantage of the fact that volume degrees of freedom can be efficiently parameterized element-wise using surface degrees. Consequently, the size of the global linear system is substantially reduced compared to standard DG schemes, as the only global unknowns are those defined on the mesh skeleton. This reduction alleviates the challenges posed by the increased degrees of freedom in DG methods. Moreover, HDG methods allow for efficient implementation through static condensation and parallelization, making them a practical and scalable choice for high-frequency wave propagation problems.

The literature presents several approaches for the a priori error analysis of HDG methods for the Helmholtz equation. In \cite{griesmaier2011error}, the authors prove optimal convergence rates for the HDG discretization of the interior Dirichlet problem under the constraint that $\omega^2 h$ is sufficiently small. However, the dependence of the error estimates on the wave number $\omega$. In \cite{chen2013hybridizable}, an $hp$-HDG method was developed for the Helmholtz equation with an impedance boundary condition and high wave number. This work accounted for the dependence on $\omega$ in the a priori error estimates but obtained a suboptimal convergence rate for the error in the flux variable. In \cite{cui2013analysis}, the authors addressed the Helmholtz equation with a Robin boundary condition, providing error bounds that included positive powers of $h^{-1}$, making them non-optimal. In 2021, \cite{zhu2021preasymptotic} presented a preasymptotic error analysis for the high-frequency case with an impedance boundary condition. This analysis focused on the linear case and achieved optimal convergence rates when $\omega^3 h^2$ is sufficiently small. The error estimates in \cite{chen2013hybridizable} and \cite{cui2013analysis} hold without imposing mesh constraints. However, \cite{cui2013analysis} assumes a star-shaped domain, in contrast to \cite{chen2013hybridizable}, \cite{griesmaier2011error}, and \cite{zhu2021preasymptotic}, which assume convex domains. The last assumption is necessary to ensure sufficient regularity when using a duality argument to establish the error estimates.

Recently, \cite{zhu_wu_2024} proposed a $(p,p-1)$-HDG method for the Helmholtz problem with a high wavenumber. However, this method is unsuitable for our purposes, as the $L^2$-error of the flux converges as $h^p$ rather than as the optimal rate $h^{p+1}$. This limitation arises from using approximations of degree $p-1$ instead of the typical choice of a piecewise polynomial approximation of degree $p$ for the flux.

Computer hardware and software advances enable numerical simulations in 2D and 3D domains. However, these simulations often have significant drawbacks, including high energy consumption and long computation times. These challenges are typically mitigated by adapting the mesh size through local refinements guided by a posteriori error estimators. The design and analysis of such estimators have been an active area of research for the past 40 years (see, e.g., \cite{ainsworth_oden_2000a, demkowicz_2006a, verfurth} and references therein).

A posteriori error estimates have been well established for finite element discretizations of second-order elliptic problems \cite{ainsworth_oden_2000a,verfuhrt1996review}. The earliest works on a posteriori error estimates for the Helmholtz equation focused on linear finite elements for one-dimensional problems \cite{babuvska1997posteriori,babuska_ihlenburg_strouboulis_gangaraj_1997}. These studies showed that in the asymptotic regime $\omega^2 h \ll 1$, both the residual-based \cite{verfurth_2013} and the Zienkiewicz-Zhu \cite{zz_1987} a posteriori error estimators are reliable and efficient, meaning they provide upper and lower bounds for the discretization error, up to a constant. Further works, including a posteriori error estimation for goal-oriented and $hp$-adaptive strategies, can be found in \cite{darrigrand_pardo_muga_2015,peraire_patera_1999,stewart_hughes_1996a,stewart_hughes_1996b,stewart_hughes_1997,stewart_hughes_1997b} and the references therein. Subsequent works have proposed and analyzed residual-based a posteriori error estimators for high-order FEM and DG discretizations of $d$-dimensional ($d=2,3$) problems. In \cite{dorfler_sauter_2013a,sauter_zech_2015a}, error upper bounds were established, even in the pre-asymptotic regime, for problems with piecewise polynomial source and boundary data. Residual minimization techniques were employed in \cite{monsuur2023pollution} to discretize a well-posed ultra-weak first-order variational formulation. This work's a posteriori error estimator was derived as an optimal test norm, defined using the bounded and invertible first-order partial differential operator, and measured in a canonical norm.

In the context of HDG, starting with the seminal works \cite{MR2914423,MR3033028} for diffusion problems, few contributions have been made in a posteriori error analysis. In addition to the above contributions associated with residual-type error estimators, \cite{ainsworth_fu_2018} proposes error estimators based on equilibrated fluxes. For more complex equations, we highlight (residual-based) a posteriori error analyses for convection-diffusion \cite{chen_li_qiu_2016}, Stokes/Brinkman \cite{araya_solano_vega_2019a}, Oseen \cite{araya_solano_vega_2019b}, Navier-Stokes \cite{leng_2021}, and Maxwell \cite{chaumontfrelet_vega_2024} (see also \cite{chen_qiu_shi_2018} for its coercive version) problems, among others.

An alternative to directly constructing a posteriori error indicators is to design finite element approximations based on residual minimization in discrete dual norms using an enriched discrete test space. This approach forms a key foundation of Discontinuous Petrov-Galerkin (DPG) methods introduced in \cite{DemGopCMAME2010} (see also \cite{DemGopBOOK-CH2014} for an overview), one of the most popular methods based on residual minimization. A standard procedure for solving the residual minimization problem employs a mixed formulation, allowing for the simultaneous computation of the discrete solution and its discrete residual representative \cite{CohDahWelM2AN2012}. The main advantage of these procedures is that one naturally incorporates the built-in residual representative to estimate the true error and guide adaptive mesh refinements \cite{DPG3}.

In \cite{calo2020adaptive}, the authors introduced a residual minimization-based finite element method relying on nonconforming discrete dual norms. Similar to DPG, this approach enables the simultaneous computation of a conforming solution and a residual representative in the nonconforming discrete space, which can be leveraged for automated mesh adaptivity. Within this framework, finite element methods based on residual minimization have been extended to address problems such as advection-diffusion-reaction \cite{cier2021automatically} and incompressible flow \cite{kyburg2022incompressible, los2021dgirm}.

More recently, \cite{muga_rojas_vega_2022a} proposed an adaptive superconvergent finite element method for solving partial differential equations involving a diffusion term. This approach combines a superconvergent postprocessing technique \cite{stenberg_1991} with a local residual minimization strategy to handle mixed formulations effectively. While various strategies exist to enhance convergence through postprocessing \cite{tofu_manuel_2024,stenberg_1991}, we adopt the method in \cite{muga_rojas_vega_2022a} as it not only provides a postprocessed approximation but also yields an a posteriori error indicator as a natural byproduct.

The approach proposed in \cite{muga_rojas_vega_2022a} begins with a discrete mixed formulation designed to approximate a model problem. This formulation is the basis for a postprocessing scheme targeting the scalar variable. Subsequently, the residual associated with the postprocessing equation is minimized and measured in a norm defined on the dual of the (discrete) test space used in the postprocessing scheme.

In this work, we focus on an HDG method for the Helmholtz equation \cite{zhu2021preasymptotic}, in contrast to the BDM discretization \cite{BDDF, BDM} of the Poisson equation considered in \cite{muga_rojas_vega_2022a}, by adapting the postprocessing scheme in \cite{muga_rojas_vega_2022a} to formulations tailored for the Helmholtz equation. The main objective of our work is to improve the quality of scalar variable approximation through superconvergent postprocessed approximations of the scalar variable and a posteriori error indicators to drive adaptive mesh refinements, which are computed locally and in parallel, and therefore, at low computational cost. The design and construction of the postprocessed approximation and the a posteriori error indicator can start from any HDG discretization of the Helmholtz equation for which a priori error estimates, ideally optimal, have been proven. Last is the key to obtaining superconvergence for the error in the scalar variable, for the postprocessed approximation, and an a posteriori error indicator for an error containing such a superconvergent term. We choose the method proposed in \cite{zhu2021preasymptotic} over other HDG methods due to its less restrictive assumptions about the domain and its capability to handle high-order discretizations effectively. We prove a priori error estimates for the postprocessed approximations and a posteriori error estimates using our error indicators, tracking the frequency dependence in the error estimates. To the best of the authors' knowledge, this work represents the first contribution to the a posteriori error analysis for an HDG discretization of the Helmholtz problem.

The remainder of the paper is organized as follows. In Section \ref{sec:prelim}, we introduce notations and the model problem whose solution we will approximate. An HDG method for the Helmholtz equation is introduced in Section \ref{sec:HDG}. In addition, we recall the stability and a priori error estimates for the method. Two minimum-residual postprocessed approximations and two a posteriori error indicators are introduced in Section \ref{sec:minres_alt}.
Furthermore, we state our main results: a priori error estimates for the postprocessed approximations, and a posteriori error estimates for those errors in terms of our error indicators. We provide numerical experiments to illustrate our main results in Section \ref{sec:numerics} and draw our conclusions in Section \ref{sec:conclusion}. Finally, the auxiliary and complementary independent results we use are contained in the Appendices \ref{app:aux_results} and \ref{app:apriori_minus}.

\section{Functional and discrete settings}\label{sec:prelim}

Let $\Omega \subset \mathbb{R}^d$, $d=2,3$, be a bounded Lipschitz domain with a polygonal or polyhedral boundary $\Gamma = \partial\Omega$, partitioned into two disjoint relatively open sets $\Gamma_D$ and $\Gamma_R$. This manuscript will use standard notation for Sobolev spaces $H^s(\Omega)$, including their norms and seminorms. Vector-valued functions and their corresponding spaces will be denoted in bold. Additionally, the following function spaces will be employed: $H^1_{\Gamma_D}(\Omega) = \{w \in H^1(\Omega): w|_{\Gamma_D} = 0\}$ and $\BH(\ddiv,\Omega) = \{\bv \in \BL^2(\Omega): \div \bv \in L^2(\Omega)\}$. We refer the reader to \cite{adams_fournier_2003a} for further details.

Let us now introduce a family of conforming, in the sense of \cite{ciarlet}, and shape-regular discretizations of $\Omega$, denoted by $\{\CT_h\}_{h>0}$. $\CT_h$ is a partition of $\Omega$ into nonoverlapping simplicial elements $K$, such that the nonempty intersection of two is a common vertex, a common edge in 2D, or a common face in 3D. $h_K$ denotes the diameter of each $K \in \CT_h$ and the mesh size is defined thus $h\eq\max_{K\in\CT_h}h_K$. The diameter of an edge/face $ F $ of a simplex $K$ is denoted by $h_F$. We further introduce the set of all edges/faces known as the skeleton $ \CF_h\eq\CF_h^I\cup\CF_h^{\partial}$, where $\CF_h^{\partial}\eq\CF_h^D\cup\CF_h^R$, while $\CF_h^D$ and $\CF_h^R$ are the set of boundary edges/faces that lay on the boundaries $ \Gamma_D $ and $\Gamma_R$, respectively. The remaining edges/faces are called interior and correspond to the elements of the set $ \CF_h^I $. Moreover, for given $F\in\CF_h$ and $ K\in\CT_h $, let us define the sets $\CF_K\eq\{F\in\CF_h:F\subset\partial K\}$, $\CT_F\eq\{K\in\CT_h:F\in\CF_K\}$, $ \CF_K^I\eq\CF_K\cap\CF_h^I$, $ \CF_K^D\eq\CF_K\cap\CF_h^D$, and $ \CF_K^R\eq\CF_K\cap\CF_h^R $.

We define the following discrete spaces in terms of the mesh $\CT_h$ and the skeleton $\CF_h$, 
\begin{align*}
	\BCP_k(\CT_h)&\eq \{\bv\in \BL^2(\CT_h) : \bv|_K \in \BCP_k(K) \quad\forall \ K \in \CT_h \}, \\
	\CP_k(\CT_h)&\eq \{w \in  L^2(\CT_h) : w|_K \in \CP_k(K) \quad\forall \ K \in \CT_h\}, \\
	\CP_k(\CF_h)&\eq \{\mu \in  L^2(\CF_h) : \mu|_F \in \CP_k(F) \quad\forall \ F \in \CF_h \},\\
    \CR_k(\partial\CT_h)&\eq \{w \in  L^2(\partial\CT_h) : w|_{\partial K} \in \CR_k(\partial K) \quad\forall \ K \in \CT_h\},
	\end{align*}
where $\CP_k(K)$, $\CP_k(F)$, and $\CR_k(\partial K)\eq\{w\in L^2(\partial K):w|_F\in\CP_k(F)\quad\forall F\subset\partial K\}$ denote the spaces of complex-valued polynomials of degree at most $k \in \mathbb{N}\cup\{0\}$ on $K$, $F$, and $\partial K$, respectively. Here, $\BCP_k(K) \eq [\CP_k(K)]^d$.

For a scalar-valued function $ w:\overline{\Omega}\to\mathbb{R} $, the jump of $ w $ across $ F\in\CF_h $ is defined as
\begin{align*}
	\jmp{w}_F\eq\begin{cases}
		w^+|_F-w^-|_F,&\text{if }F=\partial K^+\cap\partial K^-,\\
		w|_F,&\text{if }F\in\CF_h^{\partial},
	\end{cases}
\end{align*}
where $w^+\eq w|_{K^+}$ and $w^-\eq w|_{K^-}$. Here, $K^+ $ and $ K^- $ are such that the normal to $ F $ points from $ K^+ $ to $ K^- $.

We also define the inner products
\begin{align*}
    \left(\cdot, \cdot \right)_{\CT_h} \eq \sum_{K \in \CT_h} \left(\cdot, \cdot \right)_{K},\qquad \left(\cdot, \cdot \right)_{1,\CT_h} \eq \sum_{K \in \CT_h} \left(\cdot, \cdot \right)_{1,K},\qquad \langle\cdot, \cdot \rangle_{\partial \CT_h} \eq \sum_{K \in \CT_h} \langle\cdot, \cdot \rangle_{\partial K},
\end{align*}
together with the norms
\begin{align*}
&\| \cdot \|_{\CT_h} \eq \left( \sum_{K \in \CT_h} \| \cdot \|_K^2 \right)^{1/2},\quad\|\cdot\|_{1,\CT_h}\eq\left(\sum_{K\in\CT_h}\|\cdot\|_{1,K}^2\right)^{1/2},\quad\|\cdot\|_{1,J,\CT_h}\eq\left(\sum_{K\in\CT_h}\|\cdot\|_{1,J,K}^2\right)^{1/2},\\
&|\cdot|_{1,\CT_h}\eq\left(\sum_{K\in\CT_h}|\cdot|_{1,K}^2\right)^{1/2},\quad\text{and}\quad |\cdot|_{J,\CT_h}\eq\left(\sum_{K\in\CT_h}|\cdot|_{J,K}^2\right)^{1/2}
 \end{align*}
and, for $w\in H^1(K)$, its local versions $\|w\|_{1,K}^2\eq\omega^2\|w\|_K^2+|w|_{1,K}^2$ and $\|w\|_{1,J,K}^2\eq\|w\|_{1,K}^2+|w|_{J,K}^2$, with
\begin{align*}
    |w|_{J,K}^2\eq \sum_{F\in\CF_K^I\cup\CF_K^D}h_F^{-1}\|\!\jmp{w}\|_F^2+\sum_{F\in\CF_K^R}\omega^2 h_F\|w\|_F^2.
\end{align*}

Here, $(\cdot,\cdot)_K$, $\langle\cdot,\cdot\rangle_{\partial K}$, and $\langle\cdot,\cdot\rangle_F$ are the usual inner products on $L^2(K)$, $L^2(\partial K)$, and $L^2(F)$, respectively, and for $w,v\in H^1(K)$, we denote by $(w,v)_{1,K}\eq\omega^2(w,v)_K+(\grad w,\grad v)_K$ a frequency-weighted $H^1(K)$-inner product.

For each $K \in \CT_h$, and $r\in\mathbb{N}\cup\{0\}$, we denote by $Q^r_K : L^2(K)\to \CP_r(K)$ the $L^2$-projection from $L^2(K)$ onto $\CP_r(K)$, and $Q^r_h : L^2(\Omega)\to \CP_r(\CT_h)$ its global counterpart, that is, $(Q^r_h v)|_K = Q^r_K(v|_K)$. In particular, we set $Q_h \eq Q^0_h$ and $Q_K \eq Q^0_K$ for each $K \in \CT_h$, and we recall that the action of the operator $I - Q_K$ provides zero-mean functions on $K$.

In the remainder of this document, if $A,B\geq 0$, we will write $A \lesssim  B$, when there exists a constant $C > 0$, independent of $h$ and $\omega$, such that $A\leq C B$, to simplify the notation. We also write $A\eqsim B$ if $A\lesssim B$ and $B\lesssim A$.

\section{An HDG discretization of the Helmholtz equation}
\label{sec:HDG}

This section reviews the HDG scheme for the Helmholtz equation proposed in \cite{zhu2021preasymptotic}, with slight modifications to allow for mixed boundary conditions. While \cite{zhu2021preasymptotic} focuses on linear elements, we extend its analysis to high-order schemes to align with the framework of interest: a high-order HDG discretization with optimal and frequency-explicit a priori error estimates for the Helmholtz equation. Hereinafter, we consider $p\geq 1$.\\

\noindent
Given the complex-valued and square-integrable functions $f:\Omega\to\mathbb{C}$ and $g:\Gamma_R\to\mathbb{C}$, we seek solutions $(u, \bq) \in H^1_{\Gamma_D}(\Omega)\times\BH(\ddiv,\Omega)$, such that 
\begin{subequations}
\label{helmholtz}
\begin{align}
\bq + \grad u &= 0 \quad \text{in }\Omega,\label{eq:helm1} \\
\div \bq -\omega^2 u&= f  \quad\hspace{-0.025cm}\text{in }\Omega,\label{eq:helm2}  \\
u &= 0\quad \text{on }\Gamma_D, \label{eq:helm3}  \\
-\bq \cdot \bn + i \,  \omega u &= g\quad \text{on }\Gamma_R, \label{eq:helm4} 
\end{align}
\end{subequations}
where $i$, $\omega$, and $\bn$ denote the imaginary unit, the wavenumber, and the unit normal vector pointing outward $\Omega$, respectively.

\subsection{The HDG discretization}\label{subsec:HDG}

We consider the HDG scheme associated to \eqref{helmholtz} consisting in seeking $(\bq_h, u_h, \widehat{u}_h) \in \BCP_p(\CT_h) \times \CP_p(\CT_h) \times \CP_p(\CF_h)$, approximating $(\bq, u, u\vert_{\CF_h})$, such that\footnote{Here, we use $\langle\cdot, \cdot \rangle_{\partial \CT_h\setminus\mathcal{O}}\eq \sum_{K \in \CT_h} \langle\cdot, \cdot \rangle_{\partial K\setminus\mathcal{O}}$, with $\mathcal{O}\subset\overline{\Omega}$, and set $\langle\cdot, \cdot \rangle_{\Gamma_\star}\eq\langle\cdot, \cdot \rangle_{\CF_h^\star}$, with $\star\in\{D,R\}$.}
\begin{subequations}\label{eq:HDG_noextrarhs}
\begin{align}
(\bq_h, \bv_h)_{\CT_h} - (u_h, \div \bv_h)_{\CT_h} + \langle\widehat{u}_h, \bv_h \cdot \bn \rangle_{\partial \CT_h} &= 0,\label{eq:HDGa}\\
	- \omega^2(u_h,w_h)_{\CT_h}-(\bq_h, \grad w_h)_{\CT_h} + \langle\widehat{\bq}_h \cdot \bn, w_h\rangle_{\partial \CT_h} 
 &=  (f,w_h)_{\CT_h},\\
\langle\widehat{\bq}_h \cdot \bn, \mu_h\rangle_{\partial \CT_h \setminus \Gamma} &= 0, \\
\langle -\widehat{\bq}_h\cdot \bn + i\omega\widehat{u}_h,\mu_h\rangle_{\Gamma_R} &= \langle g,\mu_h\rangle_{\Gamma_R},\\
\langle\widehat{u}_h,\mu_h\rangle_{\Gamma_D} &= 0,
\end{align}
for all $(\bv_h, w_h, \mu_h)\in \BCP_p(\CT_h) \times \CP_p(\CT_h) \times \CP_p(\CF_h)$. Here
\begin{eqnarray}\label{eq:Relationfluxes}
    \widehat{\bq}_h \cdot \bn \eq \bq_h \cdot \bn + \tau (u_h - \widehat{u}_h) \quad \text{on } \partial \CT_h,
\end{eqnarray}
\end{subequations}
where $\tau:\partial\CT_h\to\CR_0(\partial\CT_h)$ is a stabilization function satisfying $\tau\geq 0$ and $\tau|_{\partial K}\neq 0$ for all $K\in\CT_h$.

\subsection{The HDG projection}
\label{sec:HDG_proj}

The following projection, designed to mimic the expression of the numerical flux, is pivotal to deriving the a priori error analysis of the HDG method \eqref{eq:HDG_noextrarhs}. For a given $K \in \CT_h$ and a nonnegative stabilization function $\tau$, the HDG projector $\Pi_h$ is defined as
\begin{align*}
\Pi_h : \BH^1(\CT_h) \times H^1(\CT_h) &\rightarrow \BCP_p(\CT_h) \times \CP_p(\CT_h)\\
(\bv, w) &\mapsto \Pi_h(\bv, w) \eq (\Pi_{\BV}\bv, \Pi_W w),
\end{align*}
which is characterized by the solution of the system of equations
\begin{subequations}\label{eq:proj}
    \begin{alignat}{6}\label{projectionPiV}
    (\Pi_{\BV} \bv,  \br_h)_K 
    &= (\bv, \br_h)_K&&\qquad\forall \br_h\in \BCP_{p-1}(K),\\
    (\Pi_W w, v_h)_K &= (w, v_h)_K &&\qquad\forall v_h \in \CP_{p-1}(K),\label{projectionPiW}\\
    \langle \Pi_{\BV} \bv \cdot \bn +\tau \Pi_W w, \mu_h \rangle_{\partial K} 
    &= \langle\bv \cdot \bn +\tau w, \mu_h\rangle_{\partial K}  &&\qquad\forall \mu_h \in \CR_p(\partial K).
    \label{projectionPiF}
    \end{alignat}
\end{subequations}

\noindent 
It is worth mentioning that the projector admits complex-valued functions; see \cite[Section 3.5.1]{MR3970243} for a more detailed description. The following result establishes the approximation properties of $\Pi_{\BV}$ and $\Pi_W$.

\begin{theorem}[{\cite[Proposition 3.6]{MR3970243}}]\label{thm:projection}
If $(\bv,w):K\to\mathbb{C}^d\times\mathbb{C}$ is sufficiently smooth and $\tau\in\CR_0(\partial K)$, then system \eqref{eq:proj} is uniquely solvable for $\Pi_{\BV}\bv$ and $\Pi_W w$. Furthermore, there is a constant $C$ independent of $K$, $\omega$ and $\tau$ such that
\begin{subequations}\label{ApproxPropertiesProjectors}
\begin{align}
\|\Pi_{\BV} \bv - \bv\|_{K} 
&\lesssim h_{K}^{l_{\bv}+1} |\bv|_{l_{\bv}+1,K}
+ h_{K}^{l_w+1} \, \tau_K^* |w|_{l_w+1,K},\label{ApproxPropertiesProjectors1}\\
\|\Pi_W w - w\|_K
&\lesssim h_{K}^{l_w+1} |w|_{l_w+1,K}
+ \frac{h_K^{l_{\bv}+1}}{\tau_K^{\max}} \, |\div \bv|_{l_{\bv},K}\label{ApproxPropertiesProjectors2},
\end{align}
\end{subequations}
for $l_w,l_{\bv}\in[0,k]$, $k\in\mathbb{N}$, where $\tau_K^{\max}\eq\max\tau|_{\partial K}$ and $\tau_K^*\eq\max\tau|_{\partial K\setminus F^*}$, being $F^*$ the face of $K$ at which $\tau|_{\partial K}$ is maximum.
\end{theorem}

\subsection{A priori error estimates}
\label{subsec:apriori}

A key ingredient of our a posteriori error estimator is a postprocessing of the HDG approximation. This postprocessing convergence, in turn, relies on the a priori error estimates for the HDG scheme.
	
To simplify the notation, we define the following quantities: $\boldsymbol{\varepsilon}_h^q\eq\Pi_{\BV}\bq-\bq_h\in\BCP_p(\CT_h)$ and $\varepsilon_h^u\eq\Pi_W u-u_h\in\CP_p(\CT_h)$. We also define the approximation errors $\be_h^q\eq\Pi_{\BV}\bq-\bq$ and $e_h^u\eq\Pi_W u-u$. 

The following result is a slight adaptation of \cite[Theorem 2]{zhu2021preasymptotic} to its high-order version with mixed boundary conditions. As the proof follows similarly, we omit it here.

\begin{theorem}[A priori error estimates for HDG] \label{thm:apriori}
Let $(u,\bq)$ be the solution of \eqref{helmholtz}. Let $(u_h,\bq_h) \in\CP_p(\CT_h) \times \BCP_p(\CT_h)$ solve \eqref{eq:HDG_noextrarhs} and suppose $\tau\eqsim\omega$. If $u\in H^s(\Omega)$, then there exists a positive constant $C_0<1$ independent of $\omega$, $h$, $p$, and $\tau$, such that if $\omega^3 h^2\leq C_0$, it hold:
\begin{align}
\label{eq:apriori_u_final}
    \omega\|u-u_h\|_{\CT_h}&\lesssim
        \left((\omega h)^{r+1}+\omega(\omega h)^{r+1}\right)\!C_u,\\
    \label{eq:apriori_q_final}
    \|\bq-\bq_h\|_{\CT_h}&\lesssim
        \left((\omega h)^{r+1}+\omega(\omega h)^{r+2}\right)\!C_u,\\
    \omega\|\epsilon_h^u\|_{\CT_h}&\lesssim((\omega h)^{r+1}+\omega(\omega h)^{r+1})C_u,    
\end{align}
where $C_u\eq\omega^{-(r+1)}|u|_{r+2,\Omega}$ and $r=\min\{s-2,p\}$. Additionally, if $\Omega$ is convex, it holds
\begin{align}
\label{eq:apriori_u_final_convex}
    \omega\|u-u_h\|_{\CT_h}&\lesssim
        \left((\omega h)^{r+1}+\omega(\omega h)^{r+2}\right)\!C_u,\\
    \label{eq:apriori_q_final2}
    \|\bq-\bq_h\|_{\CT_h}&\lesssim
        \left((\omega h)^{r+1}+\omega(\omega h)^{r+2}\right)\!C_u,\\
    \omega\|\epsilon_h^u\|_{\CT_h}&\lesssim((\omega h)^{r+2}+\omega(\omega h)^{r+2})C_u.  
\end{align}
\end{theorem}

\section{Postprocessing schemes and a posteriori error indicators based on local residual minimization}
\label{sec:minres_alt}

We aim to compute an enhanced approximation of the scalar variable using postprocessing techniques with superconvergence properties and adaptive mesh refinements. To achieve this, we adopt a discrete dual residual minimization approach, which provides a superconvergent postprocessed approximation and yields a localizable discrete Riesz representative. This representative enables the definition of a posteriori error estimators, facilitating effective adaptive mesh refinement.

We introduce, for each $K\in\CT_h$, the postprocessed approximations $\nu_h^\pm\in\CP_{p+1}(\CT_h)$, defined through the following residual minimization problems (see \cite{muga_rojas_vega_2022a}):
\begin{align}\label{eq:resmin}
    \nu_h^\pm|_K\eq&\,\argmin_{w_K\in\CP_{p+1}(K)}\frac{1}{2}\|g_K^\pm-b_K^\pm(w_K,\cdot\,)\|_{1,K,*}^2\,,
\end{align}
where $\|\cdot\|_{1,K,*}$ denotes the discrete dual\footnote{Here, $\langle\cdot,\cdot\rangle$ denotes the duality pairing between $\CP_{p+2}(K)$ and $(\CP_{p+2}(K))^*$.} norm of an element in $(\CP_{p+2}(K))^*$, that is, for $\ell\in (\CP_{p+2}(K))^*$
\begin{align*}
    \|\ell\|_{1,K,*}\eq\sup_{v_h\in \CP_{p+2}(K)}\frac{|\langle\ell,v_h\rangle|}{\|v_h\|_{1,K}}\,;
\end{align*}
while $g_K^\pm$ and $b_K^\pm$ are defined by 
\begin{align}
    b_K^\pm(w,v)\eq(\grad w,\grad v)_K\pm\omega^2(w,v)_K\qquad\text{and}\qquad g_K^\pm(v)\eq-(\bq_h,\grad v)_K\pm\omega^2(u_h,v)_K,
\end{align}
for $w,v\in H^1(K)$ and $K\in\CT_h$.

We notice that solving \eqref{eq:resmin} is equivalent to find $ (r_h^\pm,\nu_h^\pm)\in \CP_{p+2}(\CT_h)\times\CP_{p+1}(\CT_h) $ such that, for each $K\in\CT_h$, $r_K^\pm\eq r_h^\pm|_K$ and $\nu_K^\pm\eq\nu_h^\pm|_K$ solve the local linear saddle-point system (see \cite{CohDahWelM2AN2012}):
\begin{subequations}\label{eq:saddlepoint}
	\begin{align}
 (r_K^\pm,v_K)_{1,K}+b_K^\pm(\nu_K^\pm,v_K)= & \, g_K^\pm(v_K)
		&&\hspace{-2.5cm} 
		\forall v_K\in \CP_{p+2}(K),\label{eq:saddlepoint_a}\\
		b_K^\pm(w_K,r_K^\pm)\hspace{2.25cm}= & \,\, 0
		&&\hspace{-2.5cm} 
  \forall w_K\!\in\CP_{p+1}(K).\label{eq:saddlepoint_b}	
  \end{align}
\end{subequations}

The definitions of $b_K^\pm$ and $g_K^\pm$ arise from different motivations, which we detail below:
\begin{enumerate}
    \item[i)] Computing a local Helmholtz problem with $b_K^-(w,v)\eq(\grad w,\grad v)_K-\omega^2(w,v)_K$ and $g_K^-(v)\eq-(\bq_h,\grad v)_K-\omega^2(u_h,v)_K$, which is a slight modification of the one proposed in \cite{zhu2021preasymptotic},
    \item[ii)] Computing a generalization of the postprocessing scheme in \cite{muga_rojas_vega_2022a}, originally proposed by \cite{stenberg_1991}, to deal with a frequency-dependent full $H^1$-norm of the error for $u$ by considering $b_K^+(w,v)\eq(\grad w,\grad v)_K+\omega^2(w,v)_K$ and $g_K^+(v)\eq-(\bq_h,\grad v)_K+\omega^2(u_h,v)_K$ (see \cite{araya_solano_vega_2019a} for a similar approach applied to the Brinkman problem).
\end{enumerate}

The main contributions of this work are a posteriori error estimates for HDG discretizations of the Helmholtz problem, which are stated in terms of our superconvergent postprocessing approximations and a posteriori error indicators computed through a discrete dual residual minimization approach. When we refer to error, we mean the error associated with $(\bq_h,\nu_h^\pm)$ rather than $(\bq_h,u_h)$, so it is natural to ask about the (super)convergence through an a priori error analysis, for the error when considering the postprocessed approximation.

The following result provides a priori error estimates for our postprocessed approximations:

\begin{theorem}[A priori error estimates for the postprocessed approximations]\label{thm:apriori_secondpostpm}

    Let $ (u,\bq) $ be the solution of \eqref{helmholtz}, let $(u_h,\bq_h)\in\CP_p(\CT_h)\times\BCP_p(\CT_h)$ be the solution of \eqref{eq:HDG_noextrarhs}, and let $\nu_h^\pm\in\CP_{p+1}(\CT_h)$ be defined by \eqref{eq:saddlepoint}. If $u\in H^s(\Omega)$, then the following estimate holds for all $h\in(0,\pi/\omega)$:
    \begin{align}\label{eq:apriori_mr_post_fullnorm}
    \|u-\nu_h^\pm\|_{1,J,\CT_h}\lesssim((\omega h)^{r+1}+\omega(\omega h)^{r+2})C_u,
    \end{align}
    where $C_u\eq\omega^{-(r+1)}|u|_{r+2,\Omega}$ and $r=\min\{s-2,p\}$. Additionally, if $\Omega$ is convex, we have
    \begin{align}\label{eq:apriori_mr_post_L2}
        \omega\|u-\nu_h^\pm\|_{\CT_h}\lesssim((\omega h)^{r+2}+\omega(\omega h)^{r+2})C_u.
    \end{align}

Moreover, \eqref{eq:apriori_mr_post_fullnorm}  and \eqref{eq:apriori_mr_post_L2} hold for all $h>0$ in the case of $\nu_h^+$.
\end{theorem}
 
Using the solution to \eqref{eq:saddlepoint}, we define the (a posteriori) error indicators $(\eta_h^\pm)^2\eq\sum_{K\in\CT_h}(\eta_K^\pm)^2$,
where
\begin{align}\label{eq:eta}
    (\eta_K^\pm)^2\eq&\ \|r_h^\pm\|_{1,K}^2+\|\bq_h+\grad\nu_h^\pm\|_K^2+\omega^2\|u_h-\nu_h^\pm\|_K^2.
\end{align}

We also introduce the following notation for the error $(\mathrm{e}_h^\pm)^2\eq\sum_{K\in\CT_h}(\mathrm{e}_K^\pm)^2$, where
\begin{align*}
    (\mathrm{e}_K^\pm)^2\eq\|u-\nu_h^\pm\|_{1,J,K}^2+\omega^2\|u-u_h\|_K^2+\|\bq-\bq_h\|_K^2.
\end{align*}

Now, before stating our reliability result, we introduce a key assumption:

\begin{assumption}[Saturation]\label{ass:saturation}
    Let $(u,\bq) \in H^1(\Omega) \times \BH(\ddiv,\Omega)$ solve \eqref{helmholtz}, let $ (r_h^\pm,\nu_h^\pm)\in \CP_{p+2}(\CT_h)\times \CP_{p+1}(\CT_h) $ solve \eqref{eq:saddlepoint}, and let $ \theta_h^\pm\in\CP_{p+2}(\CT_h) $ solve
    \begin{align}\label{eq:aux_problem}
    b_K^\pm(\theta_h^\pm,v_K)=g_K^\pm(v_K)\qquad\forall v_K\in\CP_{p+2}(K),
\end{align}
for all $K\in\CT_h$. There exists a real number $ \delta\in[0,1) $, uniform with respect to $ h $, such that $$ \|u-\theta_h^\pm\|_{1,\CT_h}\leq\delta\|u-\nu_h^\pm\|_{1,\CT_h}. $$
\end{assumption}

Now, we can state our a posteriori error estimates, reliability, and efficiency in the following two results:

\begin{theorem}[Reliability]\label{rel}
Let $(u,\bq) \in H^1(\Omega) \times \BH(\ddiv,\Omega)$ solve \eqref{helmholtz}, let $(u_h,\bq_h) \in \CP_p(\CT_h) \times \BCP_p(\CT_h)$ solve \eqref{eq:HDG_noextrarhs}, and let $(r_h^\pm,\nu_h^\pm)\in \CP_{p+2}(\CT_h)\times \CP_{p+1}(\CT_h) $ solve \eqref{eq:saddlepoint}. 
If Assumption \ref{ass:saturation} is satisfied, then the following estimate holds for all $h>0$:
	\begin{align}\label{eq:reliability}
		\mathrm{e}_h^\pm\lesssim\frac{\sqrt{1+\omega h}}{1-\delta}\eta_h^\pm.
	\end{align}
\end{theorem}

\begin{theorem}[Efficiency]\label{eff}
	Let $(u,\bq) \in H^1(\Omega) \times \BH(\ddiv,\Omega)$ solve \eqref{helmholtz}, let $(u_h,\bq_h) \in \CP_p(\CT_h) \times \BCP_p(\CT_h)$ solve \eqref{eq:HDG_noextrarhs}, and let $(r_h^\pm,\nu_h^\pm)\in \CP_{p+2}(\CT_h)\times \CP_{p+1}(\CT_h) $ solve \eqref{eq:saddlepoint}. Then, for all $ K\in\CT_h$ and all $h>0$,
    the following estimate holds:
	\begin{align}\label{efficiency_alt}
		\eta_K^\pm\leq&\ \sqrt{6}\,\mathrm{e}_K^\pm.
	\end{align} 
\end{theorem}

The remainder of this section will be devoted to the proofs of the Theorems \ref{thm:apriori_secondpostpm}, \ref{rel}, and \ref{eff}.

\begin{remark}
    We recall that $b_K^\pm$ and $g_K^\pm$ and, as a consequence, our postprocessed approximations and error indicators can be computed starting with any HDG approximation $(u_h,\bq_h)$, even when here we choose one, in particular, \cite{zhu2021preasymptotic}, to illustrate and due to its optimal convergence rates.
\end{remark}

\subsection{Well-posedness and a priori error estimates}\label{aprioripp}
 
Since our a posteriori error estimates are stated in terms of an error involving $\|u-\nu_h^\pm\|_{1,J,\CT_h}$, it is natural that questions arise regarding the well-posedness of the postprocessed approximations and the convergence of that error term. The following result addresses the well-posedness of each problem defined by \eqref{eq:saddlepoint} and the stability estimates for $(r_h^\pm,\nu_h^\pm)$ in terms of the HDG approximation $(u_h,\bq_h)$.
\begin{theorem}[Well-posedness and stability estimates]
    Let $(u_h,\bq_h)\in\CP_p(\CT_h)\times\BCP_p(\CT_h)$ be the solution of \eqref{eq:HDG_noextrarhs}. Then, there exist a unique $(r_h^\pm,\nu_h^\pm)\in\CP_{p+2}(\CT_h)\times\CP_{p+1}(\CT_h)$ solution to \eqref{eq:saddlepoint}, and the following estimates hold for each $K\in\CT_h$:
    \begin{align*}
        \|r_h^\pm\|_{1,K}\leq \|\bq_h\|_K+\omega\|u_h\|_K\qquad\text{and}\qquad\|\nu_h^\pm\|_{1,K}\leq c_{\mathrm{stab},K}\left(\|\bq_h\|_K+\omega\|u_h\|_K\right),
    \end{align*}
    where $c_{\mathrm{stab},K}\eq\begin{cases}
        2&\text{ for }\nu_h^+,\\
        \displaystyle 2\!\left(1+\left(\frac{\omega h_K}{\pi}\right)^{\!2}\right)\!\!\left(1-\left(\frac{\omega h_K}{\pi}\right)^{\!2}\right)^{\!-1}&\text{ for }\nu_h^-.
    \end{cases}$
\end{theorem}
\begin{proof}
    The result for $(r_h^+,\nu_h^+)$ follows from the coercivity of $(\cdot,\cdot)_{1,K}$ and $b_K^+(\cdot,\cdot)$ on $(\CP_{p+2}(K),\|\cdot\|_{1,K})$ via Babu\v{s}ka-Brezzi Theorem \cite[Theorem 49.13]{ern_guermond_2021b}. For $(r_h^-,\nu_h^-)$, we decompose\footnote{For $v\in\CP_k(K)$, $k\in\mathbb{N}$, we write $v=v_\star+\overline{v}$, with $\overline{v}:=\frac{1}{|K|}\int_K v\in\CP_0(K)$ and $v_\star:=v-\overline{v}\in\CP_k^\star(K)$.} test spaces from \eqref{eq:saddlepoint_a} and \eqref{eq:saddlepoint_b} to obtain \eqref{eq:saddlepoint_a} tested with $\CP_{p+2}^\star(K)$ and $\CP_0(K)$, and \eqref{eq:saddlepoint_b} tested with $\CP_{p+1}^\star(K)$ and $\CP_0(K)$, respectively. From \eqref{eq:saddlepoint_b} tested with $\CP_0(K)$, we deduce that $r_K^+\in\CP_{p+2}^\star(K)$, and inserting this into \eqref{eq:saddlepoint_a} tested with $\CP_0(K)$, we deduce that $(\nu_h^+,1)_K=(u_h,1)_K$. Thus, we rewrite \eqref{eq:saddlepoint} equivalently as
    \begin{subequations}\label{eq:saddlepoint_equiv}
	\begin{align}
 (r_K^-,v_K)_{1,K}+b_K^-(\nu_{\star,K}^-,v_K)= & \, g_K^-(v_K)
		&&\hspace{-2.5cm} 
		\forall v_K\in \CP_{p+2}^\star(K),\label{eq:saddlepoint_a_equiv}\\
		b_K^-(w_K,r_K^-)\hspace{2.5cm}= & \,\, 0
		&&\hspace{-2.5cm} 
  \forall w_K\!\in\CP_{p+1}^\star(K),\label{eq:saddlepoint_b_equiv}
  \end{align}
  where we recover $\nu_h^-$ from $\nu_{\star,h}^-$ as $\nu_h^-=\nu_{\star,h}^-+\frac{1}{|K|}\int_K  u_h$. Here, $(\cdot,\cdot)_{1,K}$ is clearly coercive on the space $(\CP_{p+2}(K),\|\cdot\|_{1,K})$ and for $b_K^-(\cdot,\cdot)$, thanks to Poncar\'e inequality $\displaystyle\|v\|_K\leq\frac{h_K}{\pi}\|\grad v\|_K$ for all $v\in H_\star^1(K)$, it is clear that\footnote{Here, $H^1_\star(K)\eq\{w\in H^1(K):(w,1)_K=0\}$ and $\CP_k^\star(K)\eq\CP_k(K)\cap H^1_\star(K)$, for $k\in\mathbb{N}$.}
\begin{align*}
    (\grad v,\grad v)_K-\omega^2(v,v)_K\geq\!\left(1-\left(\frac{\omega h_K}{\pi}\right)^{\!\!2}\,\right)\!\!\|\grad v\|_K^2\geq\!\left(1-\left(\frac{\omega h_K}{\pi}\right)^{\!\!2}\,\right)\!\!\left(1+\left(\frac{\omega h_K}{\pi}\right)^{\!\!2}\,\right)^{\!\!-1}\!\|v\|_{1,K}^2,
\end{align*}
for all $v\in\CP_{p+2}^\star(K)$. Thus, to ensure coercivity on $(\CP_{p+2}^\star(K),\|\cdot\|_{1,K})$, we require $\omega h_K<\pi$ for all $K\in\mathcal{T}_h$. Finally, if $\omega h_K<\pi$, for all $K\in\mathcal{T}_h$, the result follows, for $(r_h^-,\nu_h^-)$, thanks to Babu\v{s}ka-Brezzi Theorem.
\end{subequations}
\end{proof} 

To prove the a priori error estimate \eqref{eq:apriori_mr_post_fullnorm}, stated in Theorem \ref{thm:apriori_secondpostpm}, we start by stating and proving an upper bound for the $\|\cdot\|_{1,K}$ part of the error, as well as an upper bound for the $\omega\|\cdot\|_K$ to prove \eqref{eq:apriori_mr_post_L2}, emphasizing their dependence on the a priori error estimates for the HDG discretization:
\begin{lemma}\label{lemma:apriori_secondpostpm}
    Let $ (u,\bq) $ be the solution of \eqref{helmholtz}, let $(u_h,\bq_h)\in\CP_p(\CT_h)\times\BCP_p(\CT_h)$ be the solution of \eqref{eq:HDG_noextrarhs}, and let $\nu_h^\pm\in\CP_{p+1}(\CT_h)$ be defined by \eqref{eq:saddlepoint}. Then, the following estimates hold for all $K\in\CT_h$ and all $h\in(0,\pi/\omega)$:
    \begin{align}
    \|u-\nu_h^\pm\|_{1,K}\lesssim&\,(1+\omega h_K)\!\left(\omega\|u-u_h\|_K+\|\bq-\bq_h\|_K+\left\|\grad\!\left(u-Q_K^{p+1}u|_K\right)\right\|_K\right)\label{eq:apriori_fullH1_mr_post_lemma}\\
        &\,+\omega\!\left\|\varepsilon_h^u\right\|_K+\omega\!\left\|u-Q_K^{p+1}u|_K\right\|_K,\nonumber\\
         \omega\|u-\nu_h^\pm\|_K\lesssim&\,\omega h_K\!\left(\omega\|u-u_h\|_K+\|\bq-\bq_h\|_K+\left\|\grad\!\left(u-Q_K^{p+1}u|_K\right)\right\|_K\right)\label{eq:apriori_L2_mr_post}\\
        &\,+\omega\!\left\|\varepsilon_h^u\right\|_K+\omega\!\left\|u-Q_K^{p+1}u|_K\right\|_K.\nonumber
    \end{align}
    Moreover, \eqref{eq:apriori_fullH1_mr_post_lemma} and \eqref{eq:apriori_L2_mr_post} hold for all $h>0$ in the case of $\nu_h^+$.
\end{lemma}

Due to the relationship between the bilinear forms $b_K^\pm$ and the inner product $(\cdot,\cdot)_{1,K}$, that is, whether or not they are equal to each other (see \cite[Proposition 3.1]{muga_rojas_vega_2022a}), we will address the proof of Lemma \ref{lemma:apriori_secondpostpm} for each case separately.

\begin{proof}[Proof of Lemma \ref{lemma:apriori_secondpostpm} for $\nu_h^-$]
    Let $ v\in\CP_{p+1}^\star(\CT_h)\eq\{w_h\in\CP_{p+1}(\CT_h):(w_h,1)_K=0\quad\forall K\in\CT_h\}$ be defined through $ v|_K=(I-Q_K)\big(Q_K^{p+1}u|_K-\nu_h^-|_K\big) $, for each $ K\in\CT_h $. To start, if we consider \eqref{eq:saddlepoint_a}, with $v_K\in\CP_{p+1}(K)\subset\CP_{p+2}(K)$, and \eqref{eq:saddlepoint_b}, we obtain
    \begin{align}\label{aux_apriori_alt}
        2\omega^2(r_h^-,v_K)_K+(\grad\nu_h^-,\grad v_K)_K&-\omega^2(\nu_h^-,v_K)_K\nonumber\\
        &=-(\bq_h,\grad v_K)_K-\omega^2(u_h,v_K)_K\quad\forall v_K\in\CP_{p+1}(K).
    \end{align}
    Thanks to Lemma \ref{lemma_eff_alt} and \eqref{aux_apriori_alt}, it follows that
    \begin{align}
		\|\grad v\|_K^2-\omega^2\|v\|_K^2&\leq\left(\left\|\grad\big(u-Q_K^{p+1}u|_K\big)\right\|_K+\|\bq-\bq_h\|_K\right)\!\|\grad v\|_K\nonumber\\
		&\qquad+\omega^2\!\left(\left\|u-Q_K^{p+1}u|_K\right\|_K+\|u-u_h\|_K+2\| r_h^-\|_K\right)\!\|v\|_K\nonumber\\
            &\leq\left(\left\|\grad\big(u-Q_K^{p+1}u|_K\big)\right\|_K+\|\bq-\bq_h\|_K\right)\!\|\grad v\|_K\label{eq:bound_v_alt}\\
		&\quad\ +\omega^2\!\left(\left\|u-Q_K^{p+1}u|_K\right\|_K+\|u-u_h\|_K\right)\!\|v\|_K\nonumber\\
            &\quad\ +2\omega\left(\|\grad(u-\nu_h^-)\|_K+\omega\|u-\nu_h^-\|_K+\|\bq-\bq_h\|_K+\omega\|u-u_h\|_K\right)\!\|v\|_K.\nonumber
    \end{align}
    Since $ v|_K\in\CP_{p+1}^\star(K) $, using \eqref{eq:bound_v_alt} and the Poincar\'e inequality \cite{Poincare_ndim,PaWe1960}, we conclude that
    	\begin{align}\label{eq:bound_v_final_alt}
		\left(1-\frac{(\omega h_K)^2}{\pi^2}\right)\!\|\grad v\|_K^2&\leq\|\grad v\|_K^2-\omega^2\|v\|_K^2\\
		&\leq\left(\left\|\grad\big(u-Q_K^{p+1}u|_K\big)\right\|_K+\frac{1}{\pi}(\omega h_K)\,\omega\!\left\|u-Q_K^{p+1}u|_K\right\|_K\right.\nonumber\\
        &\qquad\left.+\left(1+\frac{2}{\pi}\omega h_K\right)\!\|\bq-\bq_h\|_K+\frac{3}{\pi}(\omega h_K)\,\omega\|u-u_h\|_K\right.\nonumber\\
        &\qquad\left.+\frac{2}{\pi}\omega h_K\|\grad(u-\nu_h^-)\|_K+\frac{2}{\pi}(\omega h_K)\,\omega\|u-\nu_h^-\|_K\right)\!\|\grad v\|_K,\nonumber
        \end{align}
	and then
        \begin{align}
		\omega\|v\|_K&\leq \omega\frac{h_K}{\pi}\|\grad v\|_K\label{eq:bound_v_final_alt_2}\\
		&\leq\left(1-\frac{(\omega h_K)^2}{\pi^2}\right)^{\!-1}\!\frac{\omega}{\pi}h_K\!\left(\left\|\grad\big(u-Q_K^{p+1}u|_K\big)\right\|_K+\frac{1}{\pi}(\omega h_K)\,\omega\!\left\|u-Q_K^{p+1}u|_K\right\|_K\right.\nonumber\\
        &\left.\hspace{4.5cm}+\left(1+\frac{2}{\pi}\omega h_K\right)\!\|\bq-\bq_h\|_K+\frac{3}{\pi}(\omega h_K)\,\omega\|u-u_h\|_K\right.\nonumber\\
        &\left.\hspace{4.5cm}+\frac{2}{\pi}\omega h_K\|\grad(u-\nu_h^-)\|_K+\frac{2}{\pi}(\omega h_K)\,\omega\|u-\nu_h^-\|_K\right)\!.\nonumber
	\end{align}
    Let us consider the splitting
    \begin{align*}
        (u-\nu_h^-)|_K=v|_K+\big(u|_K-Q_K^{p+1}u|_K\big)+Q_K\big(Q_K^{p+1}u|_K-\nu_h^-|_K\big),
    \end{align*}
    for all $K\in\CT_h$. On the one hand, thanks to \eqref{eq:bound_v_final_alt}, we have that
    \begin{align*}
        \|\grad(u-\nu_h^-)\|_K&\leq\|\grad v\|_K+\left\|\grad(u-Q_K^{p+1}u)\right\|_K\\
        &\leq\left(1-\frac{(\omega h_K)^2}{\pi^2}\right)^{\!-1}\!\!\left(\frac{2}{\pi}\omega h_K\|\grad(u-\nu_h^-)\|_K+\frac{2}{\pi}(\omega h_K)\,\omega\|u-\nu_h^-\|_K\right.\\
        &\left.\hspace{3.75cm}+\left(1+\frac{3}{\pi}\omega h_K\right)\!T_K\right)\\
        &\quad\ +\left\|\grad(u-Q_K^{p+1}u|_K)\right\|_K
    \end{align*}
    with
    \begin{align*}
        T_K\eq\left\|\grad\big(u-Q_K^{p+1}u|_K\big)\right\|_K+\omega\!\left\|u-Q_K^{p+1}u|_K\right\|_K+\|\bq-\bq_h\|_K+\omega\|u-u_h\|_K.
    \end{align*}
    Assuming $\omega h_K\leq c$, with $c\in(0,\pi)$, it follows that
    \begin{align*}
        \|\grad(u-\nu_h^-)\|_K\leq&\,\frac{2\pi c}{\pi^2-c^2}\|\grad(u-\nu_h^-)\|_K+\frac{2\pi c}{\pi^2-c^2}\omega\|u-\nu_h^-\|_K+\frac{\pi^2}{\pi^2-c^2}\!\left(1+\frac{3}{\pi}\omega h_K\right)\!T_K\\
        &\,+\left\|\grad(u-Q_K^{p+1}u|_K)\right\|_K,
    \end{align*}
    and choosing $c=(\sqrt{5}-2)\pi$, we obtain\footnote{Our choose of $c$ implies that the constant multiplying $\|\grad(u-\nu_h^-)\|_K$ in the right-hand side takes the value $1/2$.}
    \begin{align}
        \|\grad(u-\nu_h^-)\|_K\leq \omega\|u-\nu_h^-\|_K+\frac{\pi}{c}\!\left(1+\frac{3}{\pi}\omega h_K\right)\!T_K+2\left\|\grad(u-Q_K^{p+1}u|_K)\right\|_K.\label{eq:bound_semi_apriori}
    \end{align}
    On the other hand, thanks to \eqref{eq:bound_v_final_alt_2}, the bound $\omega h_K\leq c$, and \eqref{eq:bound_semi_apriori}, we obtain
    \begin{align*}
        \omega\|u-\nu_h^-\|_K\leq&\ \frac{c}{\pi}\omega\|u-\nu_h^-\|_K+\omega\!\left(1+\frac{3}{\pi}\omega h_K\right)\!h_K T_K+\omega\!\left\|u-Q_K^{p+1}u|_K\right\|_K\\
        &\ +\omega\!\left\|Q_K\!\left(Q_K^{p+1}u|_K-\nu_h^-|_K\right)\right\|_K,
    \end{align*}
    and thus
    \begin{align}
        \omega\|u-\nu_h^-\|_K\leq&\left(1-\frac{c}{\pi}\right)^{-1}\!\left(\omega\!\left(1+\frac{3}{\pi}\omega h_K\right)\!h_K T_K+\omega\!\left\|u-Q_K^{p+1}u|_K\right\|_K\right.\label{eq:bound_L2_apriori}\\
        &\hspace{2.1cm}\left.+\omega\!\left\|Q_K\!\left(Q_K^{p+1}u|_K-\nu_h^-|_K\right)\right\|_K\right).\nonumber
    \end{align}
    Testing \eqref{eq:saddlepoint} with $v_K=w_K\in\CP_0(K)$, we deduce that $Q_K\nu_h^-=Q_K u_h$, and thanks to \eqref{eq:averages} we have that $ Q_K Q_K^{p+1}u|_K=Q_K \Pi_{W}u|_K $. Thus, by using the boundedness of $Q_K$, we conclude that
    \begin{align}
        \left\|Q_K\!\left(Q_K^{p+1}u|_K-\nu_h^-|_K\right)\!\right\|_K=\left\|Q_K\!\left(\Pi_{W}u|_K-u_h|_K\right)\right\|_K\leq\left\|\Pi_{W}u|_K-u_h\right\|_K,\label{eq:bound_QK}
    \end{align}
    and, as a consequence of \eqref{eq:bound_semi_apriori}, \eqref{eq:bound_L2_apriori} and \eqref{eq:bound_QK}, we get
    \begin{align}
        \omega\|u-\nu_h^\pm\|_K\lesssim&\,\omega h_K\!\left(\omega\|u-u_h\|_K+\|\bq-\bq_h\|_K+\left\|\grad\!\left(u-Q_K^{p+1}u|_K\right)\right\|_K\right)\label{eq:apriori_L2_mr_post_lemma}\\
        &+\omega\!\left\|\varepsilon_h^u\right\|_K+\omega\!\left\|u-Q_K^{p+1}u|_K\right\|_K,\nonumber
    \end{align}
    and \eqref{eq:apriori_L2_mr_post} follows. Finally, \eqref{eq:apriori_fullH1_mr_post_lemma} follows from \eqref{eq:bound_semi_apriori} and \eqref{eq:apriori_L2_mr_post_lemma}.
\end{proof}

For $u-\nu_h^+$, we test \eqref{eq:saddlepoint_a} with $v_K\in\CP_{p+1}(K)$ and use \eqref{eq:saddlepoint_b} together with the fact that $b_K^+(\cdot,\cdot)\equiv(\cdot,\cdot)_{1,K}$, to conclude that
\begin{align}\label{eq:aux_Stenberg+}
    b_K^+(\nu_K^+,v_K)=g_K^+(v_K)\qquad\forall v_K\in\CP_{p+1}(K).
\end{align}
Furthermore, we can even prove that the solution to \eqref{eq:aux_Stenberg+} and the solution to the saddle-point problem \eqref{eq:saddlepoint}, for an appropriate $r_K^+$, are the same (see \cite[Proposition 3.1]{muga_rojas_vega_2022a}). The estimate for $u-\nu_h^+$ can be proved with arguments similar to those of \cite[Theorem 2.5]{muga_rojas_vega_2022a}, with the norm induced by $(\cdot, \cdot)_{1,K}$ in the role of the $H^1$-seminorm, so we omit its proof.

Now, to complete the auxiliary results that we need to prove Theorem \ref{thm:apriori_secondpostpm}, we state and prove upper bounds for the jump terms in the error:

\begin{lemma}[Upper bounds for the jump error terms]\label{lemma:apriori_secondpostpm_jumps}
    \begin{align}
        h_F^{-1/2}\|\jmp{u-\nu_h^\pm}\|_F\lesssim&\,\sum_{K\in\CT_F}\left(\|\bq-\bq_h\|_K+\|\grad(u-\nu_h^\pm)\|_K\right)\label{eq:apriori_jumpID_post}
        \intertext{for all $F\in\CF_h^I\cup\CF_h^D$, and}
        \omega h_F^{1/2}\|u-\nu_h^\pm\|_F\lesssim&\,\left(\omega\|u-\nu_h^\pm\|_K\right)^{1/2}\left(\omega\|u-\nu_h^\pm\|_K+(\omega h_K)\|\grad(u-\nu_h^\pm)\|_K\right)^{1/2}\label{eq:apriori_jumpR_post}
    \end{align}
    for all $F\in\CF_h^R$ and $K\in\CT_F$.
\end{lemma}

\begin{proof}
    To prove \cref{eq:apriori_jumpID_post}, we introduce $P_h$, the $L^2$-orthogonal projection into the space $\CP_0(\CF_h)$. Then, thanks to \cite[Lemmas 3.4 and 3.5]{MR3167450}, we have, respectively
    \begin{align}
        h_F^{-1/2}\|P_h\jmp{\nu_h^\pm}\|_F&\lesssim\sum_{K\in\CT_F}\|\bq_h+\grad\nu_h^\pm\|_K,\label{CZL3.4}\\
        h_F^{-1/2}\|(I-P_h)\jmp{\nu_h^\pm}\|_F&\lesssim\sum_{K\in\CT_F}\|\grad(u-\nu_h^\pm)\|_K\label{CZL3.5}
    \end{align}
    for all $F\in\CF_h^I\cup\CF_h^D$, and therefore \eqref{eq:apriori_jumpID_post} follows from \eqref{CZL3.4}, \eqref{CZL3.5}, and the triangle inequality.

    Finally, \eqref{eq:apriori_jumpR_post} is a consequence of multiplying by $\omega h_F$ the multiplicative trace inequality \cite[Lemma 12.15]{ern_guermond_2021a}: \textit{for all $v\in H^1(K)$, all $K\in\CT_h$, all $F\in\CF_K$, and all $h>0$, the following estimate holds}
    \begin{align}\label{eq:mult_trace}
        \|v\|_F\lesssim\|v\|_K^{1/2}\left(h_K^{-1/2}\|v\|_K^{1/2}+\|\grad v\|_K^{1/2}\right),
    \end{align}
    with $v=u-\nu_h^\pm$, for all $F\in\CF_h^R$.
\end{proof}

We have all the ingredients to prove Theorem \ref{thm:apriori_secondpostpm}.

\begin{proof}[Proof of Theorem \ref{thm:apriori_secondpostpm}]
    Estimates \eqref{eq:apriori_mr_post_fullnorm} and \eqref{eq:apriori_mr_post_L2} are a direct consequence of the approximation properties of $Q_h^{p+1}$ (see \cite[Theorem 18.16]{ern_guermond_2021a}) together with Theorem \ref{thm:apriori}, and Lemmas \ref{lemma:apriori_secondpostpm} and \ref{lemma:apriori_secondpostpm_jumps}.
\end{proof}

It is worth mentioning that both postprocessing schemes will enjoy the same advantages as the scheme in \cite{muga_rojas_vega_2022a}, in the sense of low computational effort in their computation via parallelization thanks to the local nature of their definitions and its superconvergence for sufficiently regular exact solutions.

Regarding the relationship between the counterparts\footnote{Considering $b_K^-$ and $g_K^-$ instead of $b_K^+$ and $g_K^+$.} of \eqref{eq:aux_Stenberg+} and \eqref{eq:saddlepoint}, their solutions do not coincide. Thus, the a priori error analysis for the counterpart of \eqref{eq:aux_Stenberg+} is out of the scope of our work. Nevertheless, we have included the a priori for this case in Appendix \ref{app:apriori_minus} for completeness.

\begin{remark}[On the jump terms in the error]
    The technical results in \cite[Lemmas 3.4 and 3.5]{MR3167450} allowed us to simplify the definition of the error indicators compared to similar approaches \cite{muga_rojas_vega_2022a}. The proof of the bound \eqref{CZL3.4}, from \cite[Lemma 3.4]{MR3167450}, relies on three key ingredients
\begin{enumerate}
\item[(a)] The original discretization must contain an equation of the type
\begin{align*}
    (\bq_h, \bv_h)_{\CT_h} - (u_h, \div \bv_h)_{\CT_h} + \langle\mu_h, \bv_h \cdot \bn \rangle_{\partial \CT_h} = 0,
\end{align*}
where $\mu_h$ is a piecewise polynomial function such that, for each $F\in\CF_h$, $\mu_h|_F$ is single-valued.
\item[(b)] The local test space of the previous equation must contain the lowest-order Raviart-Thomas space.
\item[(c)] The local postprocessed approximation $\nu_h\in\CP_{p+1}(\CT_h)$ must satisfy $(\nu_h,1)_K=(u_h,1)_K$ for all $K\in\CT_h$.
\end{enumerate}
In turn, the proof of the bound \eqref{CZL3.5}, from \cite[Lemma 3.5]{MR3167450}, does not depend on the equation or the numerical method we are considering. Therefore, the jump bounds based on these results can be obtained for any discretization and postprocessing scheme that satisfies the above three conditions.

In our case, an HDG discretization of the Helmholtz equation, we have that the equation \eqref{eq:HDGa} has the form given in (a), its local test space $\BCP_p(K)$, $p\geq 1$, contains $RT_0(K)$, and \eqref{eq:saddlepoint} tested with $v_K\in\CP_0(K)$ and $w_K\in\CP_0(K)$ let us conclude (c).

In the case of \cite{muga_rojas_vega_2022a}, a BDM discretization for a diffusion problem, we see that \cite[equations (2.6a) and (3.1b)]{muga_rojas_vega_2022a} also satisfy the required conditions, so we could bound the jump terms in the error without the need to incorporate jump terms in the error indicator.
\end{remark}

\subsection{A posteriori error estimates}\label{sec:aposteriori}

This section is dedicated to the proofs of the reliability and efficiency estimates associated with our a posteriori error indicators, i.e., each indicator as an upper and a (local) lower bound, up to a multiplicative constant, to the corresponding discretization error. We start by proving Theorem \ref{rel}, which provides reliability estimates.

\begin{proof}[Proof of Theorem \ref{rel}]
	Let $\alpha_K\eq\omega h_K/\pi$. Thanks to Lemmas \ref{lemma_rel_alt} and \ref{lemma_rel_alt+}, and after summing over all $ K\in\CT_h $, we get
	\begin{equation}\label{rel_potential_aux_alt}
		\|\theta_h^\pm-\nu_h^\pm\|_{1,\CT_h}\leq c(\alpha)\|r_h^\pm\|_{1,\CT_h},
	\end{equation}
    where\footnote{We notice that $0<c^-(\alpha)\leq 2$, for $\alpha\in(0,2-\sqrt{3}]$, and $\lim_{\alpha\to 0}c^-(\alpha)=\sqrt{2}$.} $c(\alpha)\eq\begin{cases}
     c^-(\alpha)\eq\sqrt{2}\frac{\sqrt{1+\alpha^2}}{|1-\alpha|}&\text{for }\eta_h^-\\
     1&\text{for }\eta_h^+
 \end{cases}$ and $\alpha\eq\omega h/\pi\in(0,1)$.
	
	\noindent After combining \eqref{rel_potential_aux_alt} with Assumption \ref{ass:saturation} and the triangle inequality, we obtain
	\begin{equation}\label{rel_potential}
		\|u-\nu_h^\pm\|_{1,\CT_h}\leq\frac{c(\alpha)}{1-\delta}\|r_h^\pm\|_{1,\CT_h}.
	\end{equation}
	Additionally, we use the triangle inequality to get
	\begin{align}		\label{eq:reliability_prev4}
		\|\bq-\bq_h\|_K&\leq \|\bq_h+\grad\nu_h^\pm\|_K+\|\grad(u-\nu_h^\pm)\|_K,\\
		\label{eq:reliability_prev4_u}
		\omega\|u-u_h\|_K&\leq \omega\|u_h-\nu_h^\pm\|_K+\omega\|u-\nu_h^\pm\|_K,
	\end{align}
    and therefore
    \begin{align}
		\widetilde{\mathrm{e}}_h^\pm\leq\frac{\sqrt{3}c(\alpha)}{1-\delta}\eta_h^\pm,\label{rel_1}
	\end{align}
    where $(\widetilde{\mathrm{e}}_h^\pm)^2\eq\sum_{K\in\CT_h}(\widetilde{\mathrm{e}}_K^\pm)^2$ and
    \begin{align}
        (\widetilde{\mathrm{e}}_K^\pm)^2\eq\|u-\nu_h^\pm\|_{1,K}^2+\|\bq-\bq_h\|_K^2+\omega^2\|u-u_h\|_K^2.\label{eq:error_aux_local}
    \end{align}
    For the jump terms, we recall \cref{eq:apriori_jumpID_post,eq:apriori_jumpR_post} to conclude
    \begin{align}
        |u-\nu_h^\pm|_{J,\CT_h}^2\lesssim(1+\omega h)\|u-\nu_h^\pm\|_{1,\CT_h}^2+\|\bq-\bq_h\|_{\CT_h}^2\leq(1+\omega h)(\widetilde{\mathrm{e}}_h^\pm)^2.\label{rel_2}
    \end{align}
    Combining \eqref{rel_1} with \eqref{rel_2}, we deduce
    \begin{align*}
        (\mathrm{e}_h^\pm)^2=(\widetilde{\mathrm{e}}_h^\pm)^2+|u-\nu_h^\pm|_{J,\CT_h}^2\lesssim(1+\omega h)(\widetilde{\mathrm{e}}_h^\pm)^2\lesssim\frac{(1+\omega h)}{(1-\delta)^2}(\eta_h^\pm)^2,
    \end{align*}
    and the result follows.
\end{proof}

Now, we proceed with the proof of Theorem \ref{eff}, highlighting that the local nature of this result follows from the local definition of the postprocessing scheme \eqref{eq:saddlepoint} instead of using localization techniques such as, for instance, the one based on bubble functions.

\begin{proof}[Proof of Theorem \ref{eff}]
	Let be $ K\in\CT_h $. Thanks to the triangle inequality, we have
	\begin{align}
		\|\bq_h+\grad\nu_h^\pm\|_K&\leq\|\bq-\bq_h\|_K+\|\grad(u-\nu_h^\pm)\|_K,\label{eq:eta2_alt}\\
		\omega\|u_h-\nu_h^\pm\|_K&\leq\omega\|u-\nu_h^\pm\|_K+\omega\|u-u_h\|_K,\label{eq:eta4_alt}
	\end{align}
and therefore, thanks to Lemma \ref{lemma_eff_alt} and equations \eqref{eq:eta}, \eqref{eq:error_aux_local}, \eqref{eq:eta2_alt}, and \eqref{eq:eta4_alt}, we deduce
\begin{align*}
    \eta_K^\pm\leq\sqrt{6}\,\widetilde{\mathrm{e}}_K^\pm\leq\sqrt{6}\,\mathrm{e}_K^\pm,
\end{align*}
and the result follows.
\end{proof}

\section{Numerical experiments}
\label{sec:numerics}

In this section, we consider two numerical experiments in two dimensions and present the results obtained by considering several settings to validate our theoretical findings and show the performance of our schemes. We consider the HDG discretization \eqref{eq:HDG_noextrarhs} as the input of our scheme. Inspired by \cite{zhu2021preasymptotic}, we consider the stabilization function $\tau=\left(\frac{\sqrt{2}}{2}+\frac{\sqrt{3}}{64}\omega h\right)\!\omega$ in \eqref{eq:Relationfluxes} in all our experiments. The $h$-adapted meshes will be generated by an iterative procedure driven by the a posteriori error indicators \eqref{eq:eta}. A standard adaptive mesh refinement procedure considers a loop whose iteration consists of the following four modules:
\begin{align*}
    \text{SOLVE}\ \rightarrow\ \text{ESTIMATE}\ \rightarrow\ \text{MARK}\ \rightarrow\ \text{REFINE}.
\end{align*}

First, we solve \eqref{eq:saddlepoint}, which requires solving \eqref{eq:HDG_noextrarhs} previously. Next, we select one of the a posteriori error indicators $\eta^\pm_h$ defined in \eqref{eq:eta}, and apply the Dörfler marking criterion \cite{dorfler1996convergent}. Specifically, we mark elements for refinement if the cumulative sum of the local indicators $\eta^\pm_K$, sorted in decreasing order, remains bounded above by a chosen fraction--set to $0.5$ in our experiments--of the global indicator $\eta^\pm_h$. Finally, we refine the mesh using a bisection-type refinement strategy \cite{bank1983some}. To quantify the relation between each a posteriori error indicator and the corresponding discretization error, we introduce the effectivity index $$\xi_h^\pm\eq\eta_h^\pm/\mathrm{e}_h^\pm.$$ All the experiments were carried out using the mesh generator Netgen \cite{schoberl1997netgen} and the Finite Element solver NGSolve \cite{schoberl2014c++}. In all experiments, we plot the $L^2$-error vs. Nel$^{1/2}$ for $u_h$ and $\nu^\pm_h$, a comparison between the true error and the corresponding error indicator vs Nel$^{1/2}$, and the effectivity index vs. Nel, where Nel denotes the number of elements. 

\subsection{Plane wave solution}\label{numexp:plane}

As a first test case, we consider a problem with a sufficiently smooth solution to illustrate the behavior of the postprocessed approximation and the error indicators in both the asymptotic and preasymptotic regimes for uniformly refined meshes. To this end, we present the results obtained for the problem defined so that its solution coincides with the plane wave
\begin{align*}
    u(x,y)=e^{i\omega\left(x\cos(\pi/8)+y\sin(\pi/8)\right)}.
\end{align*}
The problem is set on the unit square $\Omega\eq(0,1)^2$, with a Dirichlet boundary condition imposed on $\Gamma_D \eq \{0\}\times (0,1) \cup (0,1) \times \{0\}$, and the source term $f$ and boundary data $g$ are chosen accordingly. Additionally, we consider different frequencies $\omega$ and polynomial degrees $p=1,2,3$ for the approximation.

In Figure \ref{fig:L2_unif_plane}, we show the increment in the convergence rate of the $L^2$-error for both strategies. As expected, as $\omega$ increases, more elements are required to reach the resolved regime\footnote{We speak about the unresolved regime when $\frac{\omega h}{2\pi p}>1$, and resolved regime otherwise.}, where both postprocessed approximations exhibit similar behavior. Moreover, we observe optimal convergence rates, including superconvergence of the $L^2$-error associated with the minimum residual postprocessed solution in the resolved regime for both schemes. However, we observe slight variations in the unresolved regime. Notably, the curve for $\|u - \nu_h^+\|_{\CT_h}$ consistently lies below $\|u - u_h\|_{\CT_h}$, whereas this trend does not hold for $\|u - \nu_h^-\|_{\CT_h}$ (see Figures \ref{fig:L2_unif_plane_32pi-}, \ref{fig:L2_unif_plane_512pi-}, \ref{fig:L2_unif_plane_32pi+}, and \ref{fig:L2_unif_plane_512pi+}). Our current analysis does not fully explain this behavior.

\begin{figure}[h!]
	\centering
	\begin{subfigure}{0.32\textwidth}
		\centering
		\includegraphics[width=\textwidth]{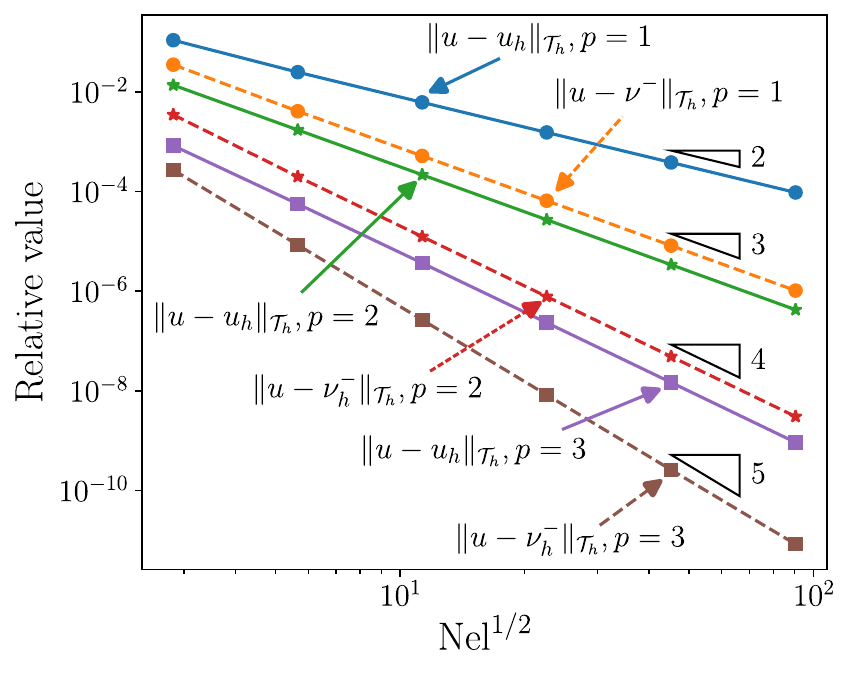}
		\caption{$\omega=\pi$, $\nu_h^-$}
		\label{fig:L2_unif_plane_pi-}
	\end{subfigure}
	\hfill
    \begin{subfigure}{0.32\textwidth}
		\centering
		\includegraphics[width=\textwidth]{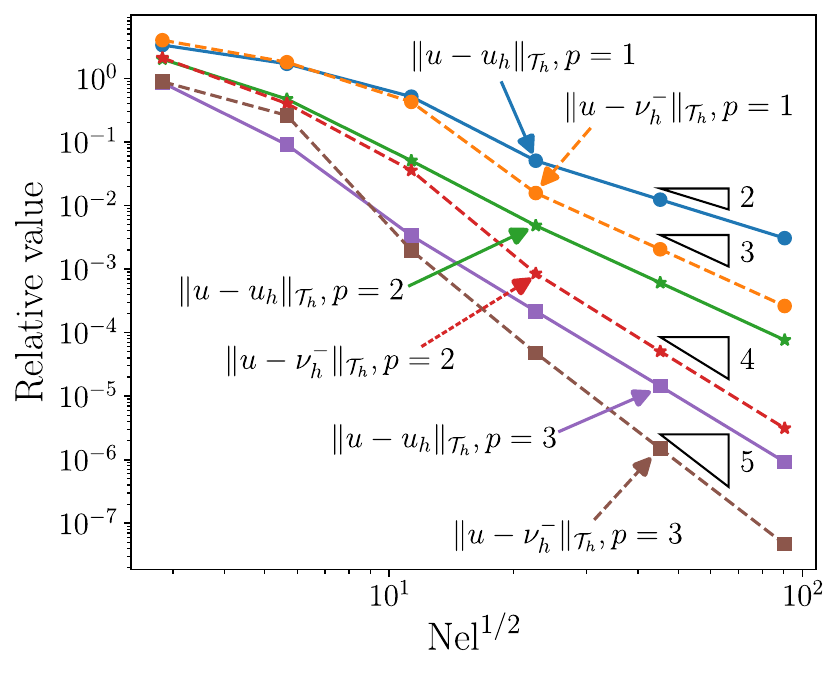}
		\caption{$\omega=4\sqrt{2}\pi$, $\nu_h^-$}
		\label{fig:L2_unif_plane_32pi-}
	\end{subfigure}
	\hfill
    \begin{subfigure}{0.32\textwidth}
		\centering
		\includegraphics[width=\textwidth]{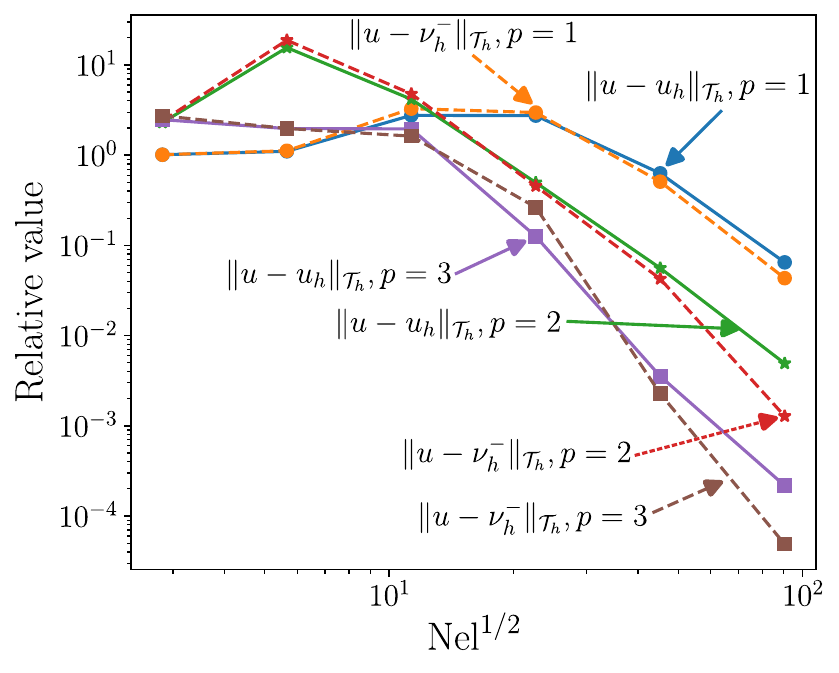}
		\caption{$\omega=16\sqrt{2}\pi$, $\nu_h^-$}
		\label{fig:L2_unif_plane_512pi-}
	\end{subfigure}
	\\
    \vspace{0.1cm}
    \begin{subfigure}{0.32\textwidth}
		\centering
		\includegraphics[width=\textwidth]{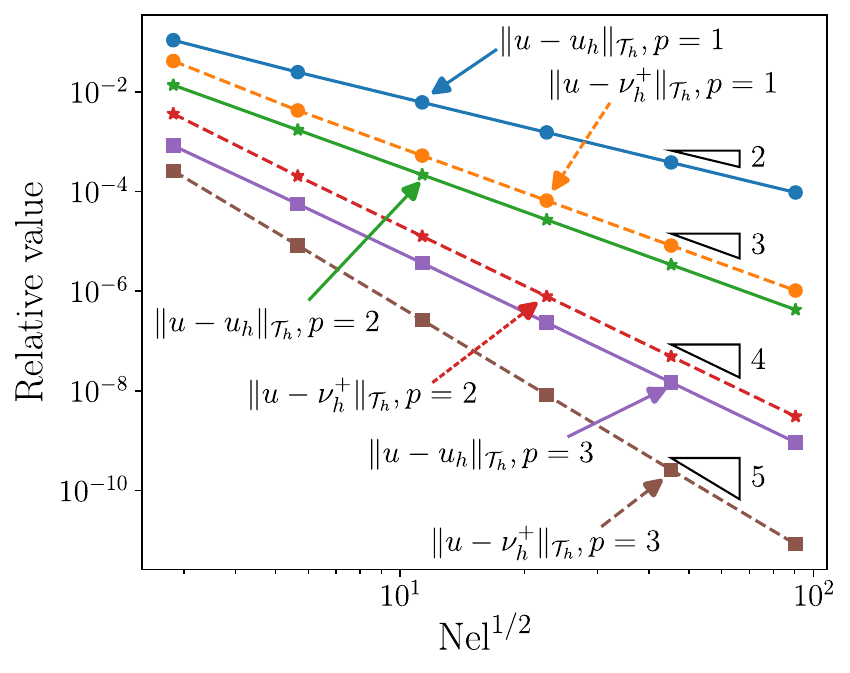}
		\caption{$\omega=\pi$, $\nu_h^+$}
		\label{fig:L2_unif_plane_pi+}
	\end{subfigure}
	\hfill
	\begin{subfigure}{0.32\textwidth}
		\centering
		\includegraphics[width=\textwidth]{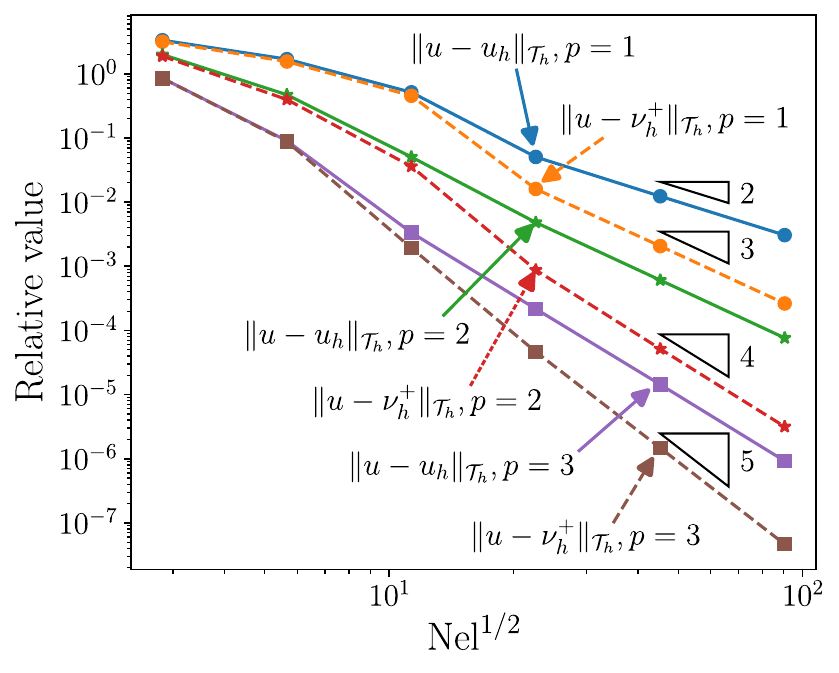}
		\caption{$\omega=4\sqrt{2}\pi$, $\nu_h^+$}
		\label{fig:L2_unif_plane_32pi+}
	\end{subfigure}
	\hfill
	\begin{subfigure}{0.32\textwidth}
		\centering
		\includegraphics[width=\textwidth]{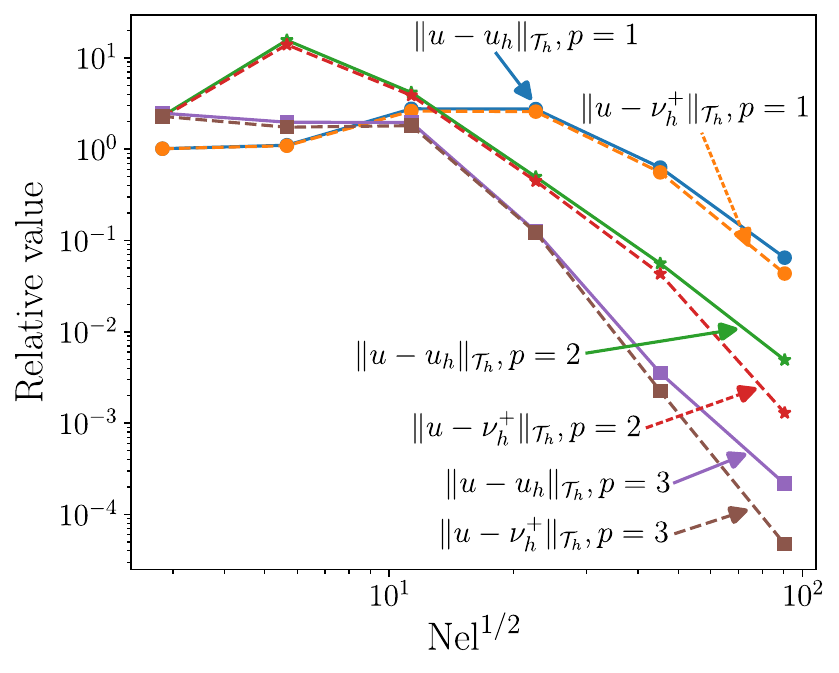}
		\caption{$\omega=16\sqrt{2}\pi$, $\nu_h^+$}
		\label{fig:L2_unif_plane_512pi+}
	\end{subfigure}
	\caption{Increment in the convergence rates of the $L^2$-error, uniform refinement, plane wave solution.}
	\label{fig:L2_unif_plane}
\end{figure}
 
Figure \ref{fig:error_unif_plane} shows the error for both postprocessed approximations measured in the norm $\|\cdot\|_{1,J\CT_h}$, together with their corresponding a posteriori error indicators. Here, we also observe similar behaviors for both schemes, with the corresponding curves (super)converging at the same rate in the resolved regime, which agrees with our theoretical findings. Indeed, we recall that the total error includes the $H^1$-norm of $u-\nu_h$ which converges as $h^{p+1}$, that is, an entire order higher than the optimal convergence rate $h^k$ for the $H^1$-norm of $u-u_h$. 

\begin{figure}[h!]
	\centering
	\begin{subfigure}{0.32\textwidth}
		\centering
		\includegraphics[width=\textwidth]{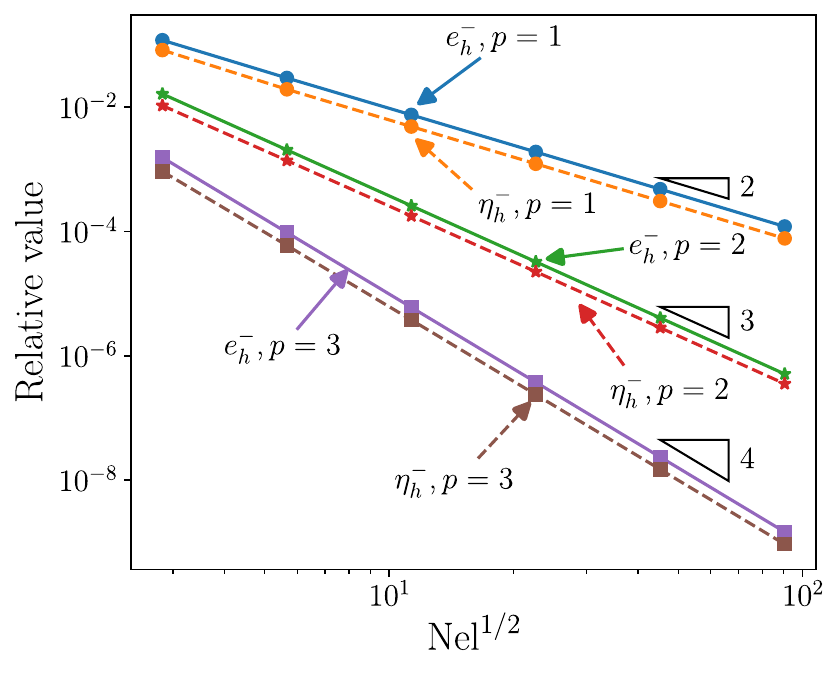}
		\caption{$\omega=\pi$, $\nu_h^-$ and $\eta_h^-$}
		\label{fig:error_unif_plane_pi-}
	\end{subfigure}
	\hfill
    \begin{subfigure}{0.32\textwidth}
		\centering
		\includegraphics[width=\textwidth]{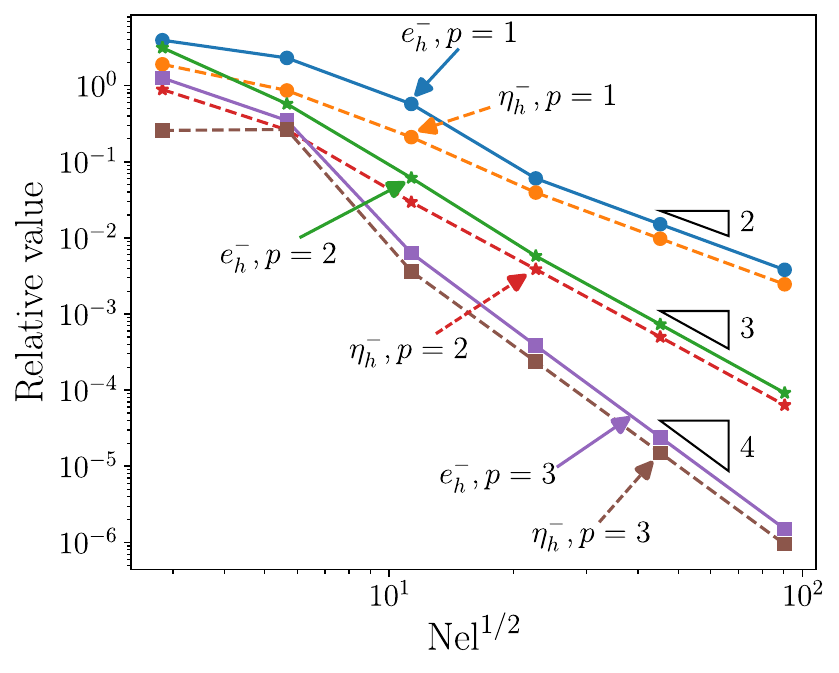}
		\caption{$\omega=4\sqrt{2}\pi$, $\nu_h^-$ and $\eta_h^-$}
		\label{fig:error_unif_plane_32pi-}
	\end{subfigure}
	\hfill
    \begin{subfigure}{0.32\textwidth}
		\centering
		\includegraphics[width=\textwidth]{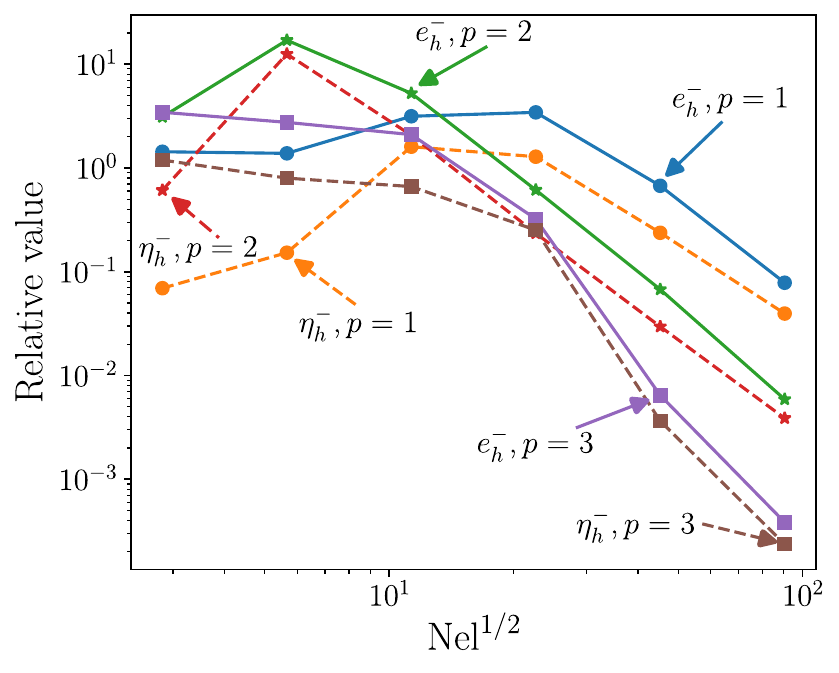}
		\caption{$\omega=16\sqrt{2}\pi$, $\nu_h^-$ and $\eta_h^-$}
		\label{fig:error_unif_plane_512pi-}
	\end{subfigure}
	\\
    \vspace{0.1cm}
    \begin{subfigure}{0.32\textwidth}
		\centering
		\includegraphics[width=\textwidth]{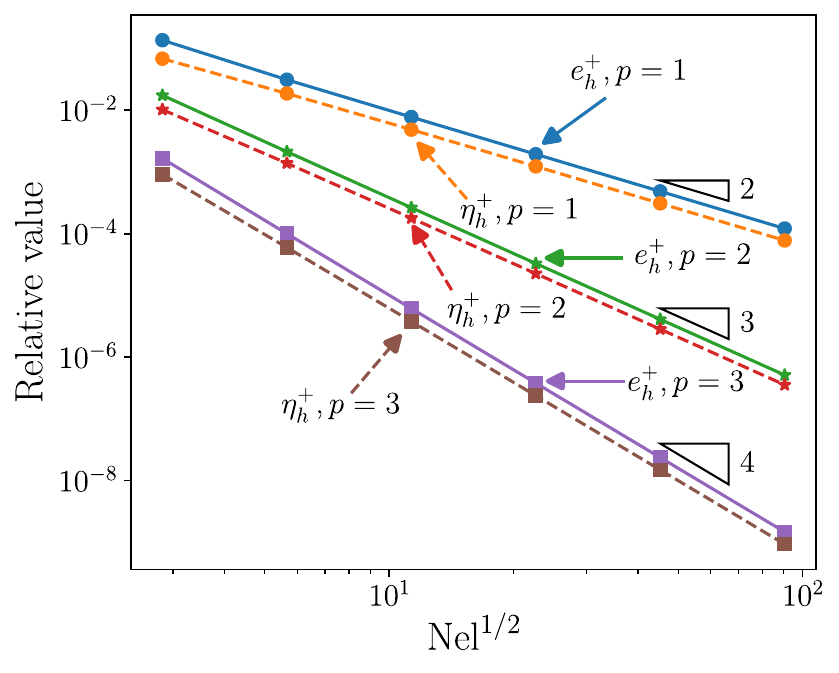}
		\caption{$\omega=\pi$, $\nu_h^+$ and $\eta_h^+$}
		\label{fig:error_unif_plane_pi+}
	\end{subfigure}
	\hfill
	\begin{subfigure}{0.32\textwidth}
		\centering
		\includegraphics[width=\textwidth]{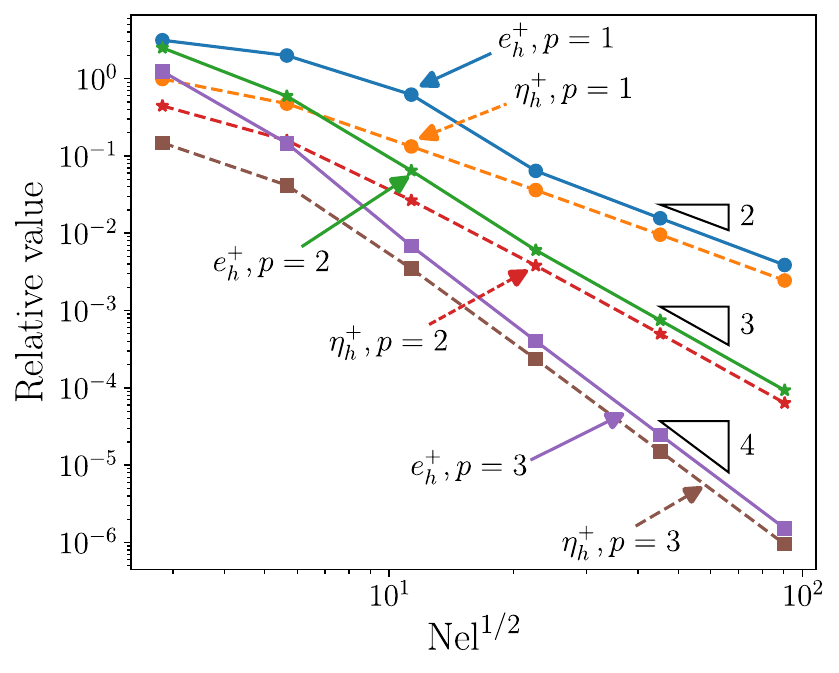}
		\caption{$\omega=4\sqrt{2}\pi$, $\nu_h^+$ and $\eta_h^+$}
		\label{fig:error_unif_plane_32pi+}
	\end{subfigure}
	\hfill
	\begin{subfigure}{0.32\textwidth}
		\centering
		\includegraphics[width=\textwidth]{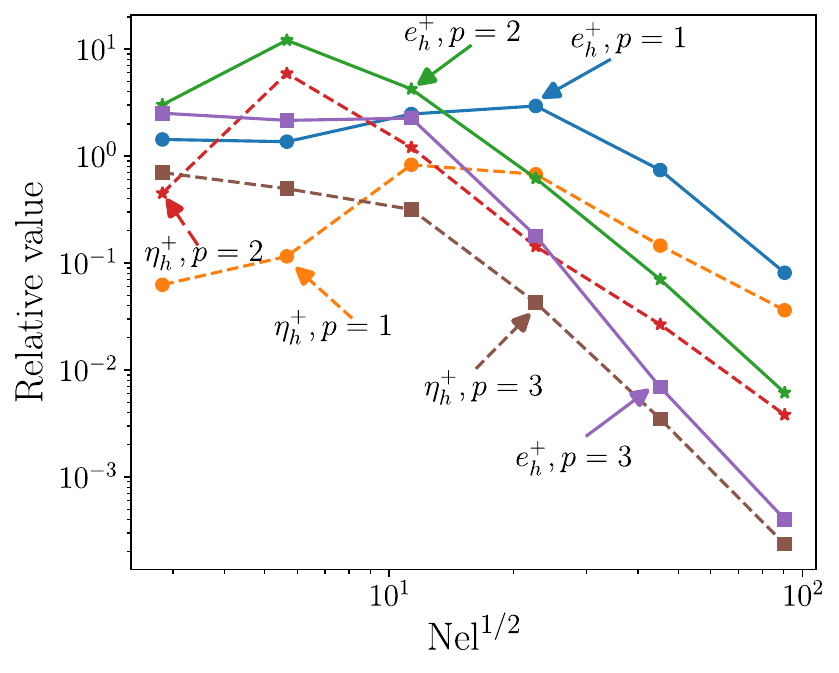}
		\caption{$\omega=16\sqrt{2}\pi$, $\nu_h^+$ and $\eta_h^+$}
		\label{fig:error_unif_plane_512pi+}
	\end{subfigure}
	\caption{$\mathrm{e}_h^\pm$ and $\eta_h^\pm$ vs. Nel$^{1/2}$ , uniform refinement, plane wave solution.}
	\label{fig:error_unif_plane}
\end{figure}

Finally, Figure \ref{fig:eff_unif_plane} presents the corresponding effectivity indices $\xi_h^\pm$ plots. As expected, we observe that the effectivity index remains bounded for a fixed $p$, with no significant differences between the schemes in the resolved regime. In the unresolved regime, however, we notice a discrepancy: $\xi_h^-$ exhibits a more oscillatory behavior, whereas $\xi_h^+$ tends to be smaller and shows a gradually increasing trend.

\begin{figure}[h!]
	\centering
	\begin{subfigure}{0.32\textwidth}
		\centering
		\includegraphics[width=\textwidth]{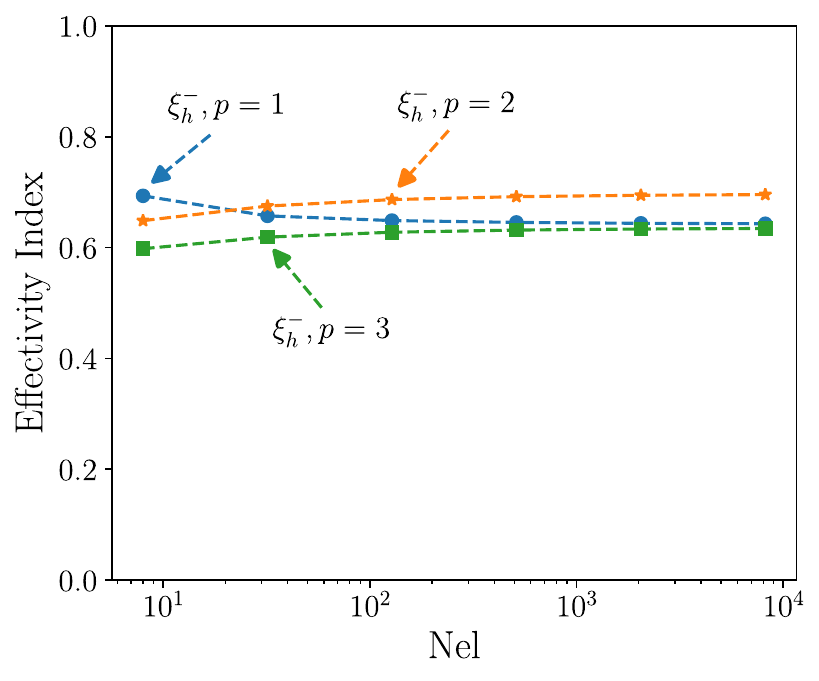}
		\caption{$\omega=\pi$, $\nu_h^-$ and $\eta_h^-$}
		\label{fig:eff_unif_plane_pi-}
	\end{subfigure}
	\hfill
    \begin{subfigure}{0.32\textwidth}
		\centering
		\includegraphics[width=\textwidth]{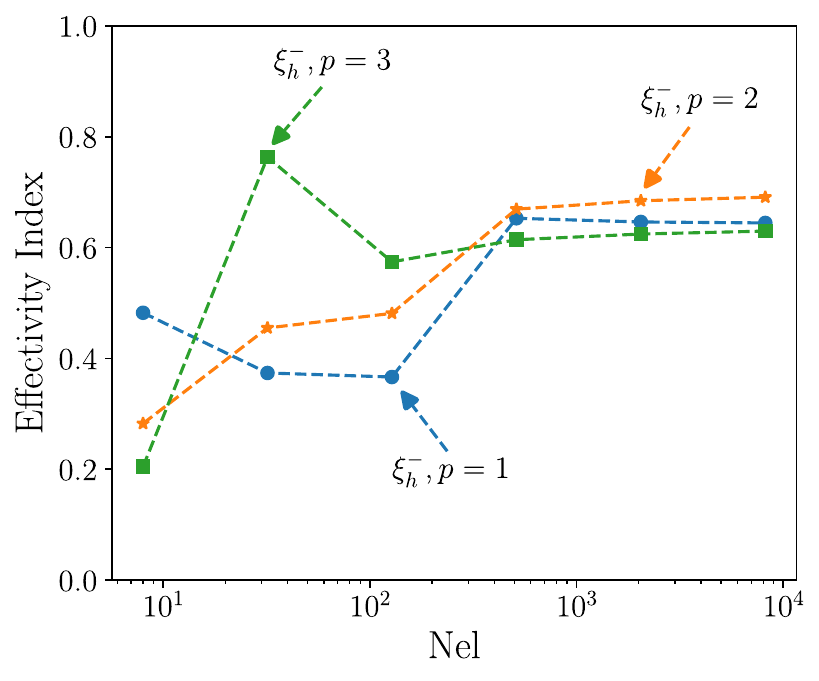}
		\caption{$\omega=4\sqrt{2}\pi$, $\nu_h^-$ and $\eta_h^-$}
		\label{fig:eff_unif_plane_32pi-}
	\end{subfigure}
	\hfill
    \begin{subfigure}{0.32\textwidth}
		\centering
		\includegraphics[width=\textwidth]{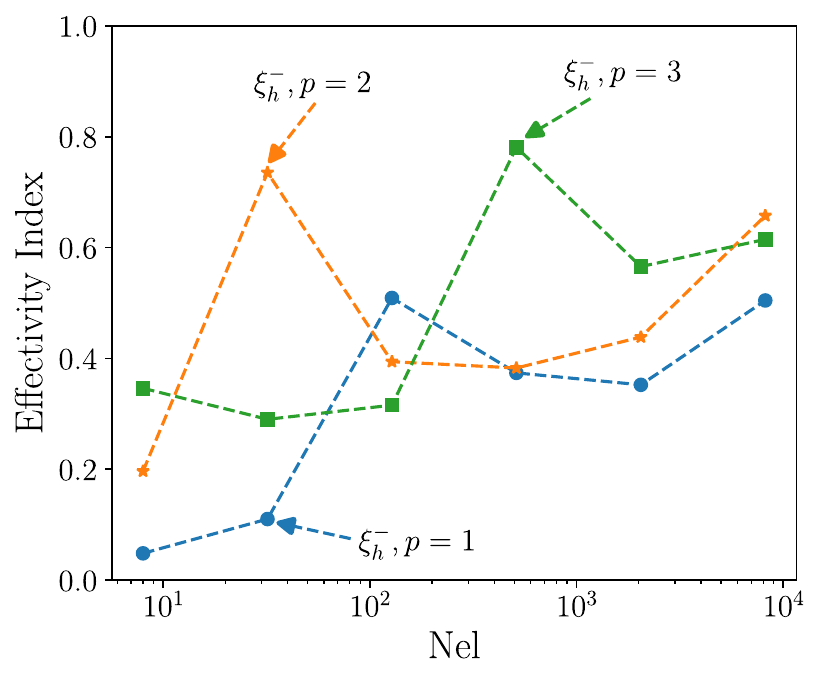}
		\caption{$\omega=16\sqrt{2}\pi$, $\nu_h^-$ and $\eta_h^-$}
		\label{fig:eff_unif_plane_512pi-}
	\end{subfigure}
	\\
    \vspace{0.1cm}
    \begin{subfigure}{0.32\textwidth}
		\centering
		\includegraphics[width=\textwidth]{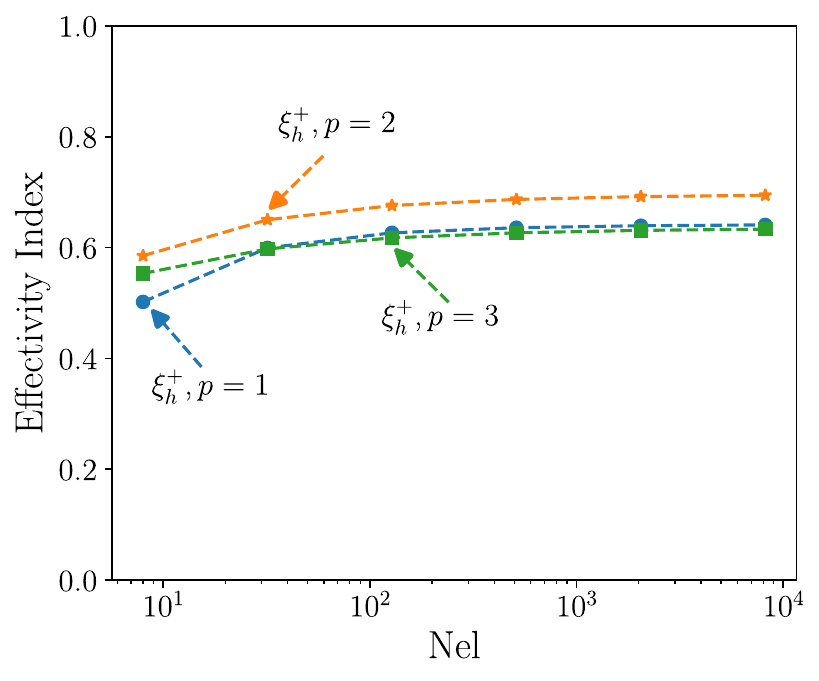}
		\caption{$\omega=\pi$, $\nu_h^+$ and $\eta_h^+$}
		\label{fig:eff_unif_plane_pi+}
	\end{subfigure}
	\hfill
	\begin{subfigure}{0.32\textwidth}
		\centering
		\includegraphics[width=\textwidth]{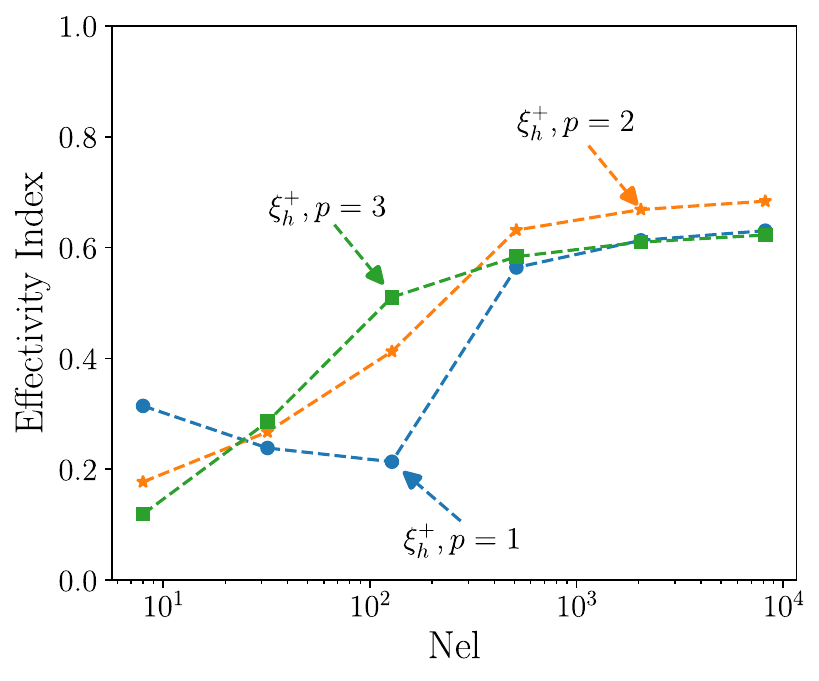}
		\caption{$\omega=4\sqrt{2}\pi$, $\nu_h^+$ and $\eta_h^+$}
		\label{fig:eff_unif_plane_32pi+}
	\end{subfigure}
	\hfill
	\begin{subfigure}{0.32\textwidth}
		\centering
		\includegraphics[width=\textwidth]{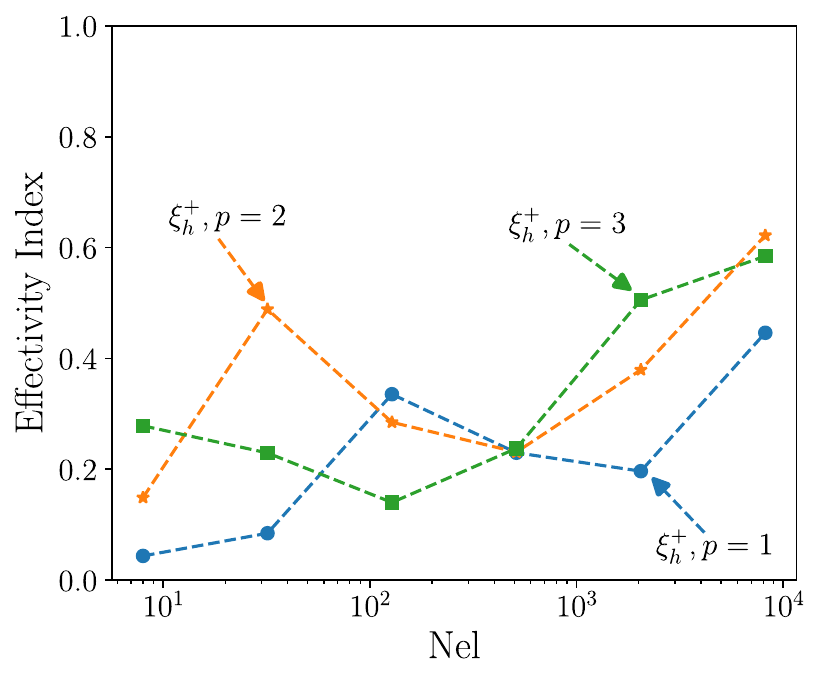}
		\caption{$\omega=16\sqrt{2}\pi$, $\nu_h^+$ and $\eta_h^+$}
		\label{fig:eff_unif_plane_512pi+}
	\end{subfigure}
	\caption{Effectivity index $\xi_h^\pm$ vs. Nel, uniform refinement, plane wave solution.}
	\label{fig:eff_unif_plane}
\end{figure}

\subsection{Singular solution}

To evaluate the performance of adaptive mesh refinement driven by our estimator, we consider a singular problem. Specifically, we set the Helmholtz equation \eqref{helmholtz} on the L-shaped domain $\Omega = (-1,1)^2 \setminus [0,1]^2$, where the exact solution in polar coordinates $(r, \varphi)$ is given by  
$$
u(r, \varphi) = \mathcal{J}_{2/3} (\omega r) \sin\left(\frac{2}{3}(\pi - \varphi)\right),
$$  
where $\mathcal{J}_{2/3}$ denotes the Bessel function of the first kind, $ f \equiv 0 $, and $ g $ is chosen accordingly. It is well known that this problem exhibits suboptimal convergence rates when using uniform mesh refinements (see, e.g., \cite[Example 4]{li_liu_yang_2023}). 

\begin{figure}[h!]
	\centering
	\begin{subfigure}{0.32\textwidth}
		\centering
		\includegraphics[width=\textwidth]{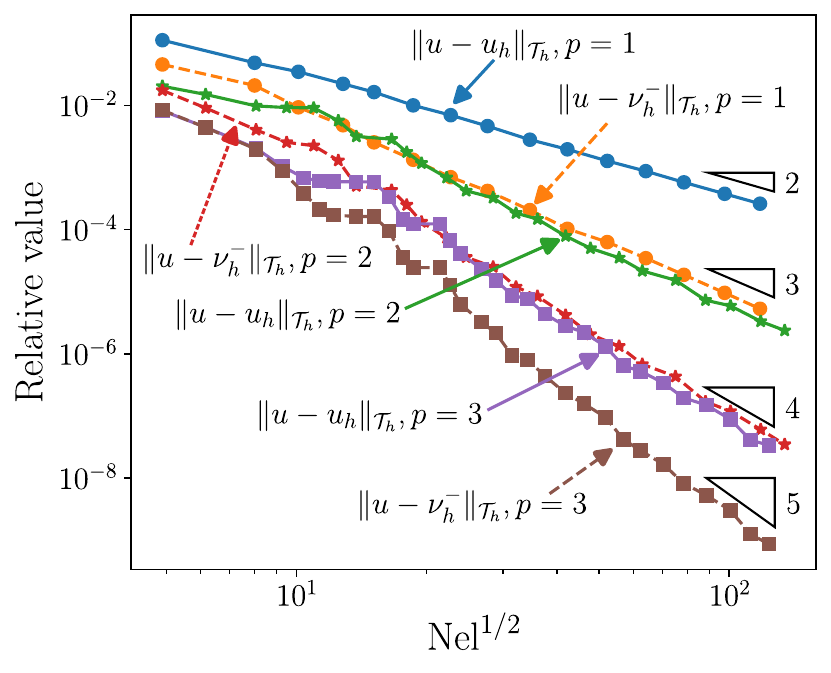}
		\caption{Enhancement of $L^2$-error}
		\label{fig:L2_adap_singular_pi-}
	\end{subfigure}
	\hfill
    \begin{subfigure}{0.32\textwidth}
		\centering
		\includegraphics[width=\textwidth]{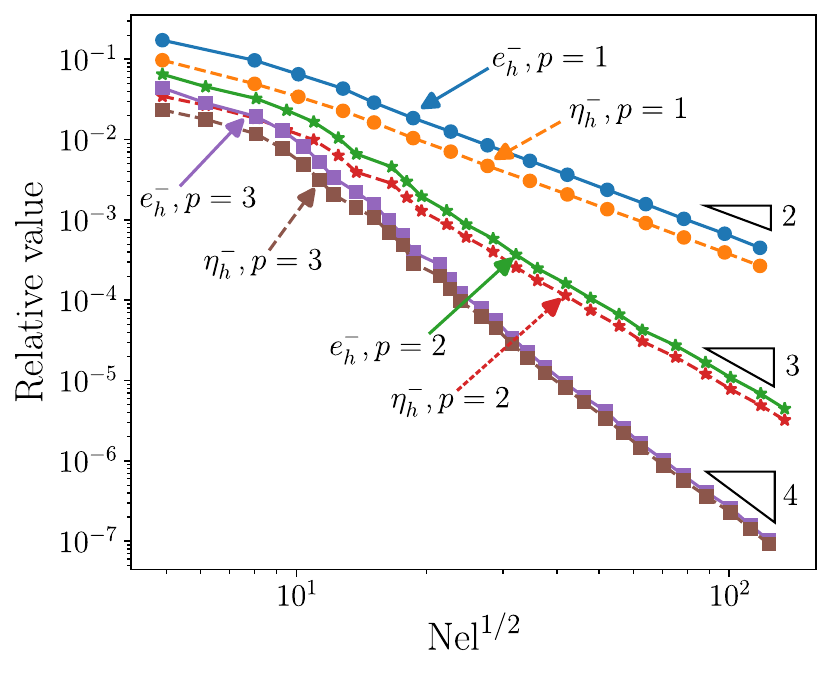}
		\caption{$\mathrm{e}_h^-$ and $\eta_h^-$ vs. $\mathrm{Nel}^{1/2}$}
		\label{fig:err_adap_singular_pi-}
	\end{subfigure}
	\hfill
    \begin{subfigure}{0.32\textwidth}
		\centering
		\includegraphics[width=\textwidth]{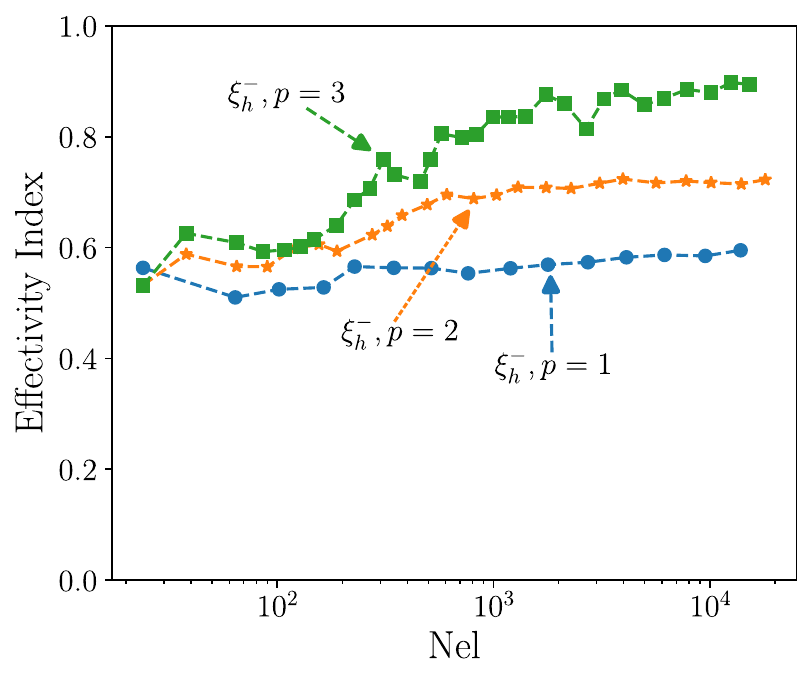}
		\caption{Effectivity index $\xi_h^-$}
		\label{fig:eff_adap_singular_pi-}
	\end{subfigure}

	\caption{$\omega=\pi$, $\nu_h^-$, and $\eta_h^-$; adaptive refinement, singular solution.}
	\label{fig:adap_singular_low}
\end{figure}

In Figure \ref{fig:adap_singular_low}, we summarize the results for a low-frequency case, considering the preprocessed solution $\nu_h^-$ and the a posteriori error indicator $\eta_h^-$. We observe that the adaptivity enhances the convergence rates, recovering the expected rates for smooth solutions (cf. Section \ref{numexp:plane}). We observe a similar behavior when considering the corresponding counterparts $\nu_h^+$ and $\eta_h^+$, so we omit these plots.

\begin{figure}[h!]
	\centering
       \begin{subfigure}{0.32\textwidth}
		\centering
		\includegraphics[width=\textwidth]{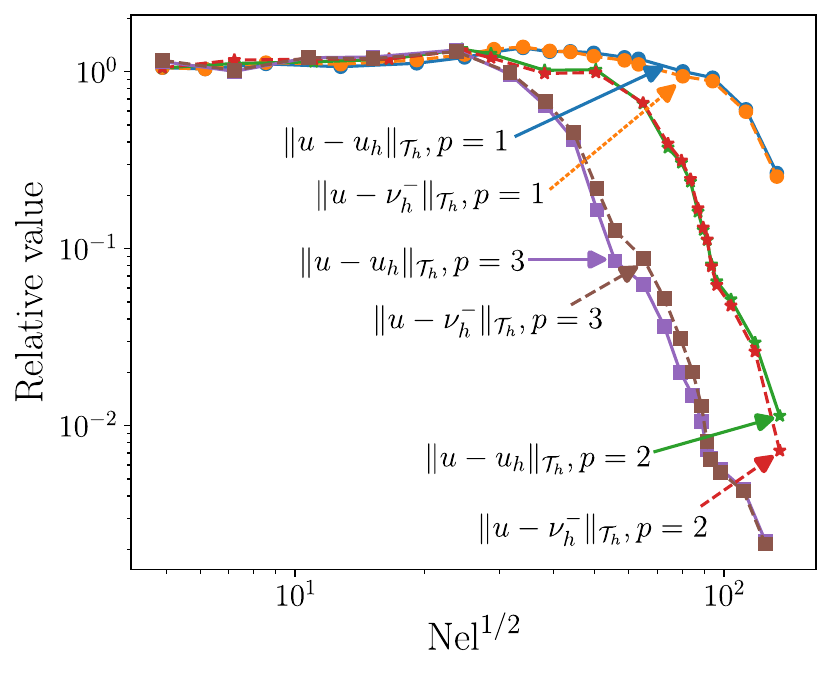}
		\caption{Enhancement of $L^2$-error}
		\label{fig:L2_adap_singular_512pi-}
	\end{subfigure}
	\hfill
	\begin{subfigure}{0.32\textwidth}
		\centering
		\includegraphics[width=\textwidth]{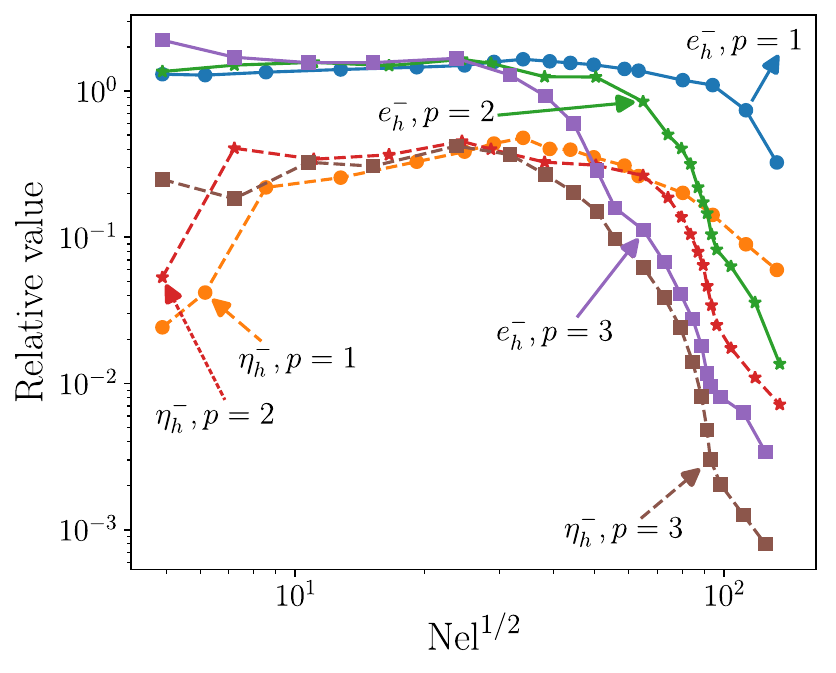}
		\caption{$\mathrm{e}_h^-$ and $\eta_h^-$ vs. $\mathrm{Nel}^{1/2}$}
		\label{fig:error_adap_singular_512pi-}
	\end{subfigure}
	\hfill
        \begin{subfigure}{0.32\textwidth}
		\centering
		\includegraphics[width=\textwidth]{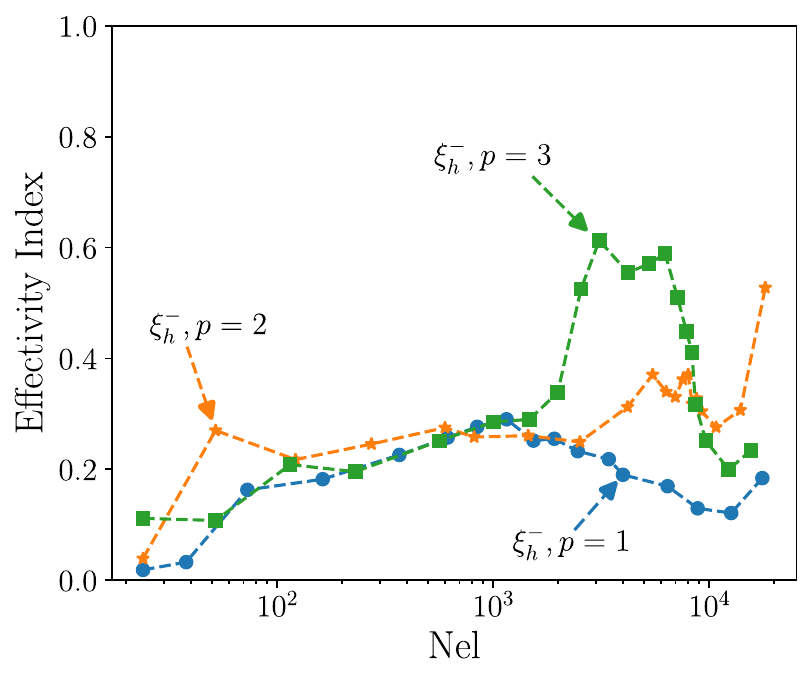}
		\caption{Effectivity index $\xi_h^-$}
		\label{fig:eff_adap_singular_512pi-}
	\end{subfigure}
	\\
    \vspace{0.1cm}
    \begin{subfigure}{0.32\textwidth}
		\centering
		\includegraphics[width=\textwidth]{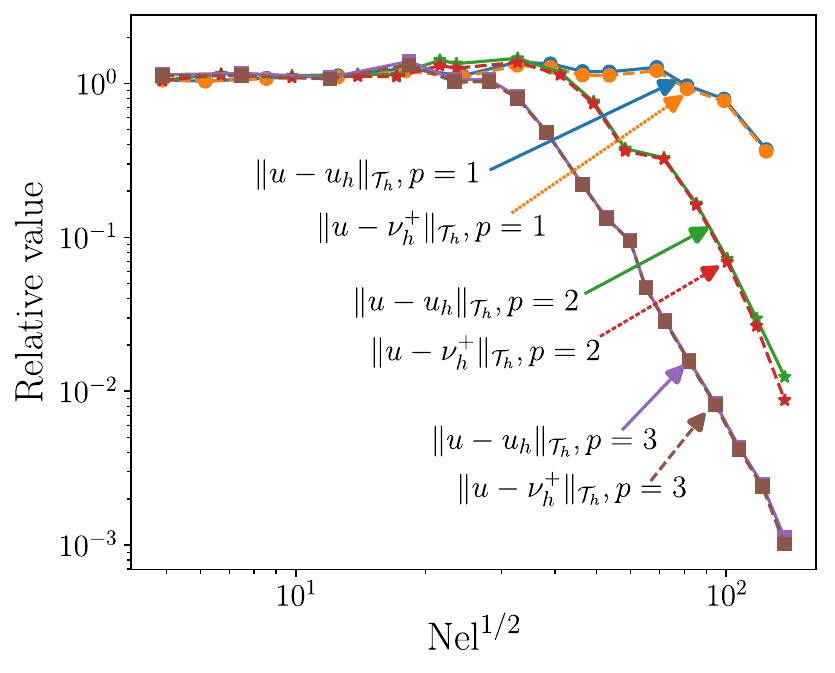}
		\caption{Enhancement of $L^2$-error}
		\label{fig:L2_adap_singular_512pi+}
	\end{subfigure}
	\hfill
        \begin{subfigure}{0.32\textwidth}
		\centering
		\includegraphics[width=\textwidth]{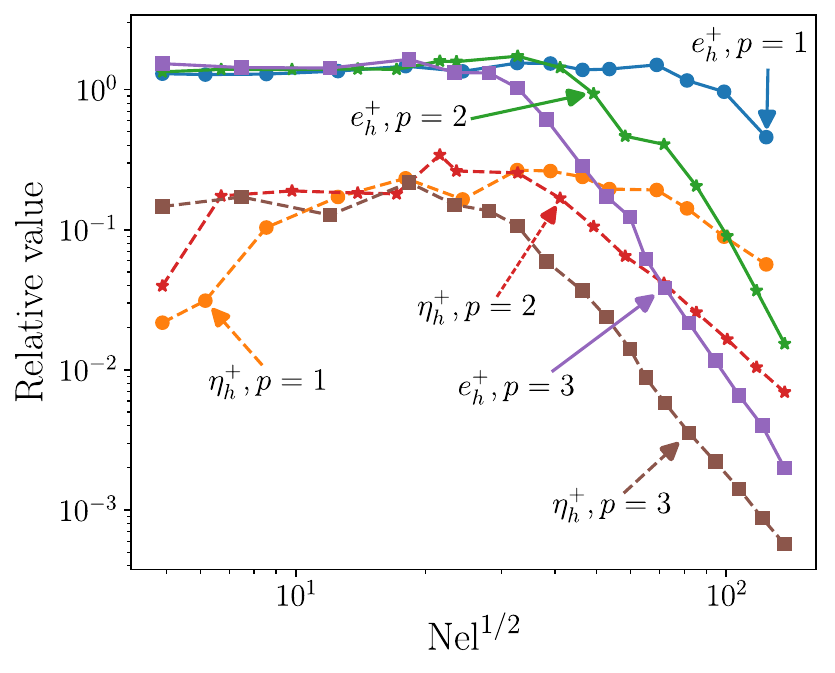}
		\caption{$\mathrm{e}_h^+$ and $\eta_h^+$ vs. $\mathrm{Nel}^{1/2}$}
		\label{fig:error_adap_singular_512pi+}
	\end{subfigure}
	\hfill
	\begin{subfigure}{0.32\textwidth}
		\centering
		\includegraphics[width=\textwidth]{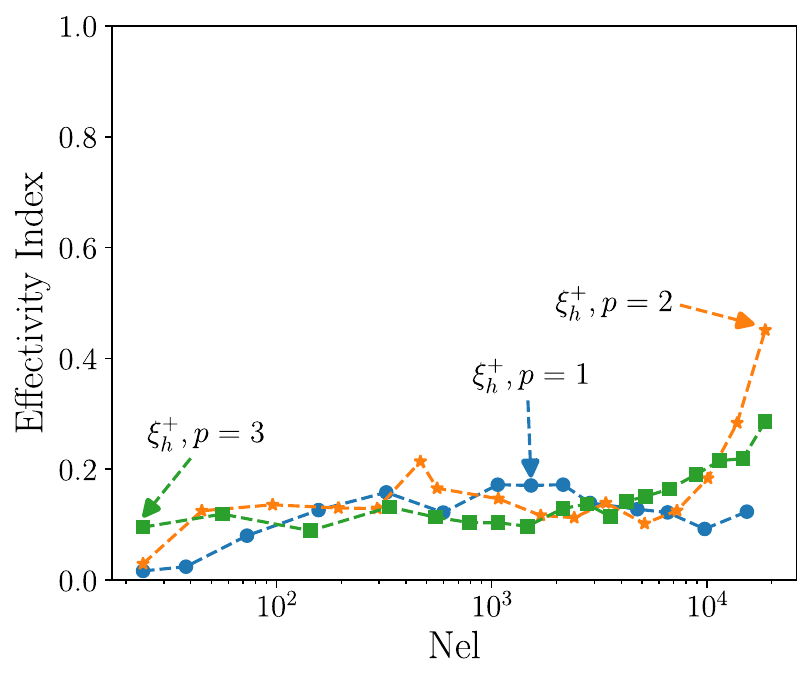}
		\caption{Effectivity index $\xi_h^+$}
		\label{fig:eff_adap_singular_512pi+}
	\end{subfigure}
	\caption{$\omega=16\sqrt{2}\pi$, $\nu_h^\pm$, and $\eta_h^\pm$; adaptive refinement, singular solution.}
	\label{fig:adap_singular_high}
\end{figure}

Finally, Figure \ref{fig:adap_singular_high} compares the performance of both postprocessed approximations and error indicators in a high-frequency case. Figures \ref{fig:L2_adap_singular_512pi-} and \ref{fig:L2_adap_singular_512pi+} illustrate a similar behavior for both postprocessed approximations. It is worth noting that $\|u - \nu_h^+\|_{\mathcal{T}_h}$ seems to be always below $\|u - u_h\|_{\mathcal{T}_h}$, whereas $\|u - \nu_h^-\|_{\mathcal{T}_h}$ does not exhibit the same trend. We observe a slightly superior performance of $(\eta_h^+, \mathrm{e}_h^+)$ (Figure \ref{fig:error_adap_singular_512pi+}) over $(\eta_h^-, \mathrm{e}_h^-)$ (Figure \ref{fig:error_adap_singular_512pi-}), as reflected in the earlier attainment of the convergence rates expected in the asymptotic regime. Additionally, we note a significant improvement in the stability of the effectivity index for $\xi_h^+$ (Figure \ref{fig:eff_adap_singular_512pi+}) compared to $\xi_h^-$ (Figure \ref{fig:eff_adap_singular_512pi-}), with noticeably more minor variations.

\section{Conclusions}
\label{sec:conclusion}

In this work, we proposed and analyzed two a posteriori error indicators for an HDG discretization of the Helmholtz problem with mixed boundary conditions. The key component in constructing these indicators is the residual representative obtained from a residual minimization problem.

Our methodology offers two main advantages. First, the residual minimization approach naturally provides a built-in residual representative, enabling the construction of a reliable and efficient a posteriori error indicator. Second, residual minimization is leveraged in local postprocessing schemes, allowing us to solve the associated saddle-point problems with minimal computational cost.

Additionally, our approach yields two distinct minimum-residual postprocessed approximations of the scalar variable, exhibiting superconvergence for sufficiently regular exact solutions in the resolved regime.

The proposed a posteriori error indicators can effectively guide adaptive mesh refinement procedures, addressing solution singularities and mitigating the pollution effect. Our numerical results for the two-dimensional Helmholtz equation demonstrate improved convergence rates and even superconvergence in the preasymptotic regime.

Future directions include further improvements in pollution reduction for high-frequency problems through more advanced postprocessing techniques and extensions to other time-harmonic wave propagation problems, such as Maxwell's equations.

\section*{Acknowledgments}
The work by Sergio Rojas was supported by the Chilean grant ANID FONDECYT No. 1240643. Patrick Vega's work was done in the Chilean grant ANID FONDECYT No. 3220858 framework. Liliana Camargo thanks the Office of Research from Pontificia Universidad Católica de Valparaíso for its support through the project DI Postdoctorado 2023.


\appendix

\section{Auxiliary results}\label{app:aux_results}
\renewcommand{\thesection}{A} 

The following auxiliary result is essential for establishing both the a priori error estimates and the efficiency estimate in the context of a posteriori error analysis (see Section \ref{sec:minres_alt}) for $u-\nu_h^\pm$.

\begin{lemma}\label{lemma_eff_alt}
	Let $(u,\bq) \in H^1(\Omega) \times \BH(\ddiv,\Omega)$ solve \eqref{helmholtz}, let $(u_h,\bq_h) \in \CP_p(\CT_h) \times \BCP_p(\CT_h)$ solve \eqref{eq:HDG_noextrarhs}, and let $(r_h^\pm,\nu_h^\pm)\in \CP_{p+2}(\CT_h)\times \CP_{p+1}(\CT_h) $ solve \eqref{eq:saddlepoint}. Then, the following holds:
	\begin{align}\label{eq:efficiency_alt}
		\|r_h^\pm\|_{1,K}&\leq \|\grad(u-\nu_h^\pm)\|_K+\omega\|u-\nu_h^\pm\|_K+\|\bq-\bq_h\|_K+\omega\|u-u_h\|_K\,,
	\end{align}
	for all $ K\in\CT_h$. 
\end{lemma}

\begin{proof}
	Thanks to equation~\eqref{eq:saddlepoint_a} and the fact that $ \bq+\grad u=0 $, for each $ K\in\CT_h $, we have
	\begin{align*}
		(\grad r_h^\pm,\grad v_K)_K+\omega^2(r_h^\pm,v_K)_K&=-(\bq_h,\grad v_K)_K-(\grad\nu_h^\pm,\grad v_K)_K\pm\omega^2(u_h-\nu_h^\pm,v_K)_K\\
		&=(\bq-\bq_h,\grad v_K)_K+(\grad(u-\nu_h^\pm),\grad v_K)_K\pm\omega^2(u_h-\nu_h^\pm,v_K)_K,
	\end{align*}
	for all $ v_K\in\CP_{p+2}(K)$\,. Thus, by using the Cauchy-Schwarz inequality, we get
	\begin{align*}
		(r_h^\pm,v_K)_{1,K}&=(\|\bq-\bq_h\|_K+\|\grad(u-\nu_h^\pm)\|_K)\|\grad v_K\|_K+\omega\|u_h-r_h^\pm\|_K\omega\|v_K\|_K\\
		&\leq(\|\bq-\bq_h\|_K+\|\grad(u-\nu_h^\pm)\|_K+\omega\|u_h-\nu_h^\pm\|_K)\|v_K\|_{1,K},
	\end{align*}
	and therefore
	\begin{equation}\label{eq:upper_bound_varepsK}
		\|r_h^\pm\|_{1,K}=\sup_{v_K\in \CP_{p+2}(K)}\frac{(r_h^\pm ,v_K)_{1,K}}{\|v_K\|_{1,K}}\leq\|\bq-\bq_h\|_K+\|\grad(u-\nu_h^\pm)\|_K+\omega\|u_h-\nu_h^\pm\|_K.
	\end{equation}
	Finally, we write
	\begin{equation}\label{eq:upper_bound_uh_nuh}
		\omega\|u_h-\nu_h^\pm\|_K\leq \omega\|u-\nu_h^\pm\|_K+\omega\|u-u_h\|_K,
	\end{equation}
	and \eqref{eq:efficiency_alt} follows from \eqref{eq:upper_bound_varepsK} and \eqref{eq:upper_bound_uh_nuh}.
\end{proof}

The following results, together with a saturation assumption (Assumption \ref{ass:saturation}), will allow us to obtain upper bounds to the $H^1(K)$-norm of $u-\nu_h^\pm$ (Theorem \ref{rel}). We start with $u-\nu_h^-$.

\begin{lemma}\label{lemma_rel_alt}
	Let $(u,\bq) \in H^1(\Omega) \times \BH(\ddiv,\Omega)$ solve \eqref{helmholtz}, let $(u_h,\bq_h) \in \CP_p(\CT_h) \times \BCP_p(\CT_h)$ solve \eqref{eq:HDG_noextrarhs}, let $\theta_h^- \in \CP_{p+2}(\CT_h)$
	solve \eqref{eq:aux_problem}, and let $(r_h^-,\nu_h^-)\in \CP_{p+2}(\CT_h)\times \CP_{p+1}(\CT_h) $ solve \eqref{eq:saddlepoint}. Then, for all $ K\in\CT_h $, the following holds true:
	\begin{align}\label{eq:reliability_prev_alt}
		\|\grad(\theta_h^--\nu_h^-)\|_K&\leq\sqrt{2}\left(1-\frac{(\omega h_K)^2}{\pi^2}\right)^{\!-1}\!\left(1+\frac{\omega h_K}{\pi}\right)\!\| r_h^-\|_{1,K}=\sqrt{2}\left(1-\frac{\omega h_K}{\pi}\right)^{\!-1}\!\| r_h^-\|_{1,K},\\      \label{eq:reliability_prev2_alt}
		\omega\|\theta_h^- -\nu_h^-\|_K&\leq \sqrt{2}\left(1-\frac{\omega h_K}{\pi}\right)^{\!-1}\!\frac{\omega h_K}{\pi}\| r_h^-\|_{1,K}.
	\end{align}
\end{lemma}
\begin{proof}
	Using \eqref{eq:saddlepoint_a}, we have
	\begin{align*}
		(\grad(\theta_h^- -\nu_h^-),\grad v_K)_K&=-(\grad\nu_h^-,\grad v_K)_K+(\grad\theta_h^-,\grad v_K)_K\\
		&=(r_h^-,v_K)_{\Xi_K}-\omega^2(\nu_h^-,v_K)_K+(\bq_h,\grad v_K)_K+\omega^2(u_h,v_K)_K\\
        &\quad\ +(\grad\theta_h^-,\grad v_K)_K\\
		&=(r_h^-,v_K)_{\Xi_K}+\omega^2(\theta_h^- -\nu_h^-,v_K)_K
	\end{align*}
	for all $v_K\in\CP_{p+2}(K)$. Since $v_K\in\CP_{p+2}(K)$ and $\theta_h^- -\nu_h^-\in\CP_{p+1}^\star(K)$, thanks to Poincar\'e inequality \cite{Poincare_ndim,PaWe1960}, we get
	\begin{align*}
		\|\grad(\theta_h^- -\nu_h^-)\|_K&=\sup_{v_K\in\CP_{p+2}(K)\setminus\{0\}}\frac{(\grad(\theta_h^- -\nu_h^-),\grad v_K)_K}{\|\grad v_K\|_K}\\
		&\leq\sup_{v_K\in\CP_{p+2}(K)\setminus\{0\}}\frac{(\grad r_h^-,\grad v_K)_K}{\|\grad v_K\|_K}+\omega^2\sup_{v_K\in\CP_{p+2}(K)\setminus\{0\}}\frac{(r_h^-,v_K)_K}{\|\grad v_K\|_K}\\
        &\quad\ +\omega^2\sup_{v_K\in\CP_{p+2}(K)\setminus\{0\}}\frac{(\theta_h^- -\nu_h^-,v_K)_K}{\|\grad v_K\|_K}\\
		&\leq\|\grad r_h^-\|_K+\frac{\omega h_K}{\pi}\omega\|r_h^-\|_K+\frac{\omega h_K}{\pi}\omega\|\theta_h^- -\nu_h^-\|_K\\
		&\leq\sqrt{2}\!\left(1+\frac{\omega h_K}{\pi}\right)\!\|r_h^-\|_{1,K}+\frac{(\omega h_K)^2}{\pi^2}\|\grad(\theta_h^- -\nu_h^-)\|_K
	\end{align*}
	and $\displaystyle\|\theta_h^--\nu_h^-\|_K\leq \frac{h_K}{\pi}\|\grad(\theta_h^- -\nu_h^-)\|_K$, that is
	\begin{align*}
		\|\grad(\theta_h^- -\nu_h^-)\|_K&\leq\sqrt{2}\left(1-\frac{(\omega h_K)^2}{\pi^2}\right)^{\!-1}\!\left(1+\frac{\omega h_K}{\pi}\right)\!\|r_h^-\|_{1,K}=\sqrt{2}\left(1-\frac{\omega h_K}{\pi}\right)^{\!-1}\!\|r_h^-\|_{1,K}
		\intertext{and}        \omega\|\theta_h^- -\nu_h^-\|_K&\leq \sqrt{2}\left(1-\frac{\omega h_K}{\pi}\right)^{\!-1}\!\frac{\omega h_K}{\pi}\|r_h^-\|_{1,K}.
	\end{align*}
\end{proof}
For $u-\nu_h^+$, we establish the following result, which generalizes \cite[Lemma A.2]{muga_rojas_vega_2022a}, replacing $|\cdot|_{1,K}$ with $\|\cdot\|_{1,K}$. Since its proof follows the same reasoning as that of \cite[Lemma A.2]{muga_rojas_vega_2022a}, we omit it.

\begin{lemma}\label{lemma_rel_alt+}
    Let $\theta_h^+ \in \CP_{p+2}(\CT_h)$
    solve \eqref{eq:aux_problem} and let $(r_h^+,\nu_h^+)\in \CP_{p+2}(\CT_h)\times \CP_{p+1}(\CT_h) $ solve \eqref{eq:saddlepoint}. Then, for all $ K\in\CT_h $, the following holds true:
    \begin{align}\label{eq:reliability_prev_alt+}
        \|\theta_h^+-\nu_h^+\|_{1,K}&=\|r_h^+\|_{1,K}.
    \end{align}
\end{lemma}

\renewcommand{\thesection}{Appendix B}
\section{A priori error estimates for the postprocessed approximation based on the Helmholtz operator.}\label{app:apriori_minus}
\renewcommand{\thesection}{B} 

The following result provides a priori error estimates for the frequency-dependent postprocessed solution $\widetilde{u}_h^-\in \mathcal{P}_{p+1}(\mathcal{T}_h)$, defined for each $K\in\mathcal{T}_h$ by  
\begin{align}\label{eq:postproc}  
    b_K^-(\widetilde{u}_h^-, v_K) = g_K^-(v_K) \qquad \forall v_K \in \mathcal{P}_{p+1}(K),  
\end{align}  
which is closely related to postprocessed approximations based on residual minimization (see \cite{muga_rojas_vega_2022a}).  

We highlight the explicit dependence of these estimates on the a priori error estimates for the HDG method \eqref{eq:HDG_noextrarhs}.

\begin{proposition}\label{thm:post_apriori-}
    Let $ (u,\bq)\in H^1(\Omega) \times \BH(\ddiv,\Omega) $ be the solution of \eqref{helmholtz}, let $(u_h,\bq_h)\in \CP_p(\CT_h) \times \BCP_p(\CT_h)$ be the solution of \eqref{eq:HDG_noextrarhs}, and let $\widetilde{u}_h^-\in \CP_{p+1}(\CT_h)$ be defined by \eqref{eq:postproc}. Then, the following estimates hold: 
	\begin{align}\label{apriori_pp}
		\omega\|u- \widetilde{u}_h^-\|_K\leq
		&\ 
		\frac{1}{\pi}\!\left(1-\frac{(\omega h_K)^2}{\pi^2}\right)^{\!-1}\!\omega h_K\!\left(\left\|\grad\big(u-Q_K^{p+1}u|_K\big)\right\|_K+\|\bq-\bq_h\|_K\right.\\
        &\hspace{4.5cm}\left.+\frac{\omega h_K}{\pi}\omega\!\left(\left\| u-Q_K^{p+1}u|_K\right\|_K+\|u-u_h\|_K\right)\right)\nonumber\\
		&\ +\omega\!\left\|u_h-\Pi_{W}u|_K\,\right\|_K
		+\omega\!\left\|u-Q_K^{p+1}u|_K\,\right\|_K,\nonumber\\
		\|\grad(u- \widetilde{u}_h^-)\|_K\leq
		&
		\left(1-\frac{(\omega h_K)^2}{\pi^2}\right)^{\!-1}\!\left(\left\|\grad\big(u-Q_K^{p+1}u|_K\big)\right\|_K+\|\bq-\bq_h\|_K\right.\\
        &\hspace{3.25cm}\left.+\frac{\omega h_K}{\pi}\omega\!\left(\left\|u-Q_K^{p+1}u|_K\right\|_K+\|u-u_h\|_K\right)\right)\nonumber\\
		&\ 
		+\left\|\grad(u-Q_K^{p+1}u|_K)\right\|_K,\nonumber
	\end{align}
	for all $ K\in\CT_h $ and all $h\in(0,\pi/\omega)$.
\end{proposition}

\begin{proof}
	Let $ v\in\CP_{p+1}^\star(\CT_h) $ be defined through $ v|_K=(I-Q_K)\big(Q_K^{p+1}u|_K- \widetilde{u}_h^-|_K\big) $, for each $ K\in\CT_h $. From \eqref{eq:postproc}, and the fact that $ \grad u + \bq=0 $, it follows that
	\begin{align*}
		\|\grad v\|_K^2-\omega^2\|v\|_K^2=&\left(\grad(I-Q_K)\big(Q_K^{p+1}u|_K- \widetilde{u}_h^-|_K\big),\grad v\right)_{\!K}\\
        &\ -\omega^2\left((I-Q_K)\big(Q_K^{p+1}u|_K- \widetilde{u}_h^-|_K\big),v\right)_K\\
		=&\left(\grad\big(Q_K^{p+1}u|_K- \widetilde{u}_h^-\big),\grad v\right)_{\!K}-\omega^2\left(Q_K^{p+1}u|_K- \widetilde{u}_h^-,v\right)_K\\
		=&\left(\grad Q_K^{p+1}u|_K,\grad v\right)_{\!K}+(\bq_h,\grad v)_K-\omega^2\left(Q_K^{p+1}u|_K,v\right)_K+\omega^2(u_h,v)_K\\
		=&-\left(\grad\big(u-Q_K^{p+1}u|_K\big),\grad v\right)_{\!K}-(\bq-\bq_h,\grad v)_K+\omega^2\left(u-Q_K^{p+1}u|_K,v\right)_K\\
        &\ -\omega^2(u-u_h,v)_K\,,
	\end{align*}
	and applying the Cauchy-Schwarz inequality, we obtain
	\begin{align}\label{eq:bound_v}
		\|\grad v\|_K^2-\omega^2\|v\|_K^2&\leq\left(\left\|\grad\big(u-Q_K^{p+1}u|_K\big)\right\|_K+\|\bq-\bq_h\|_K\right)\!\|\grad v\|_K\\
		&\qquad+\omega^2\!\left(\left\|u-Q_K^{p+1}u|_K\right\|_K+\|u-u_h\|_K\right)\!\|v\|_K\,.\nonumber
	\end{align}
	Since $ v|_K\in\CP_{p+1}^\star(K) $, \eqref{eq:bound_v} and the Poincar\'e inequality \cite{Poincare_ndim,PaWe1960} allow us to conclude that
	\begin{align}\label{eq:bound_v_final}
		&\left(1-\frac{(\omega h_K)^2}{\pi^2}\right)\!\|\grad v\|_K^2\nonumber\\
		&\leq\|\grad v\|_K^2-\omega^2\|v\|_K^2\\
		&\leq\left(\left\|\grad\big(u-Q_K^{p+1}u|_K\big)\right\|_K+\|\bq-\bq_h\|_K+\frac{\omega h_K}{\pi}\omega\!\left(\left\|u-Q_K^{p+1}u|_K\right\|_K+\|u-u_h\|_K\right)\right)\!\!\|\grad v\|_K,\nonumber
    \end{align}
    and then
        \begin{align}
		\omega\|v\|_K\nonumber
		&\leq\!\frac{\omega h_K}{\pi}\|\grad v\|_K\\
		&\leq\!\left(1-\frac{(\omega h_K)^2}{\pi^2}\right)^{\!-1}\!\frac{\omega h_K}{\pi}\!\bigg(\left\|\grad\big(u-Q_K^{p+1}u|_K\big)\right\|_K+\|\bq-\bq_h\|_K\\
        &\hspace{4.75cm}+\,\omega^2\frac{h_K}{\pi}\!\left(\left\|u-Q_K^{p+1}u|_K\right\|_K+\|u-u_h\|_K\right)\bigg).\nonumber
	\end{align}
	In addition, due to the definitions of the projections $Q_K$, $Q_K^{p+1}$, and $\Pi_W$, we have
	\begin{align}\label{eq:averages}
		\int_K Q_K Q_K^{p+1}u|_K=\int_K Q_K^{p+1}u|_K=\int_K u=\int_K \Pi_{W}u|_K=\int_K Q_K \Pi_{W}u|_K\,,
	\end{align}
	that is, $ Q_K Q_K^{p+1}u|_K=Q_K \Pi_{W}u|_K $. Moreover, recalling \eqref{eq:postproc}, tested with $v\in\CP_0(K)$, and considering the boundedness of $ Q_K $, we can conclude that
	\begin{align}\label{eq:proj_term}
		\big\|Q_K\big(Q_K^{p+1}u|_K- \widetilde{u}_h^-|_K\big)\!\big\|_K\!=\!\big\|Q_K\big(\Pi_{W}u|_K-u_h|_K\big)\!\big\|_K\!\leq\!\big\|\Pi_{W}u|_K-u_h\big\|_K\,.
	\end{align}
	Hence, \eqref{apriori_pp} follows from \eqref{eq:bound_v_final}, \eqref{eq:proj_term}, and the splitting
	\begin{align*}
		(u- \widetilde{u}_h^-)|_K=v|_K+Q_K\big(Q_K^{p+1}u|_K- \breve{u}_h|_K\big)+\big(u|_K-Q_K^{p+1}u|_K\big),
	\end{align*}
	for all $ K\in\CT_h $.
\end{proof}

\begin{remark}
    The postprocessed approximation $\widetilde{u}_h^-$ also verifies the estimates in Theorem \ref{thm:apriori_secondpostpm}. Therefore, $\widetilde{u}_h^-$ possesses the superconvergence we also prove for $\nu_h^\pm$ for sufficiently regular exact solutions in the resolved regime.
\end{remark}



\bibliographystyle{elsarticle-harv} 

\begin{thebibliography}{76}
\expandafter\ifx\csname natexlab\endcsname\relax\def\natexlab#1{#1}\fi
\providecommand{\url}[1]{\texttt{#1}}
\providecommand{\href}[2]{#2}
\providecommand{\path}[1]{#1}
\providecommand{\DOIprefix}{doi:}
\providecommand{\ArXivprefix}{arXiv:}
\providecommand{\URLprefix}{URL: }
\providecommand{\Pubmedprefix}{pmid:}
\providecommand{\doi}[1]{\href{http://dx.doi.org/#1}{\path{#1}}}
\providecommand{\Pubmed}[1]{\href{pmid:#1}{\path{#1}}}
\providecommand{\bibinfo}[2]{#2}
\ifx\xfnm\relax \def\xfnm[#1]{\unskip,\space#1}\fi
\bibitem[{Adams and Fournier(2003)}]{adams_fournier_2003a}
\bibinfo{author}{Adams, R.}, \bibinfo{author}{Fournier, J.},
  \bibinfo{year}{2003}.
\newblock \bibinfo{title}{Sobolev Spaces}.
\newblock \bibinfo{publisher}{Academic Press}.
\bibitem[{Ainsworth and Fu(2018)}]{ainsworth_fu_2018}
\bibinfo{author}{Ainsworth, M.}, \bibinfo{author}{Fu, G.},
  \bibinfo{year}{2018}.
\newblock \bibinfo{title}{Fully computable \textit{a posteriori} error bounds
  for hybridizable discontinuous {G}alerkin finite element approximations}.
\newblock \bibinfo{journal}{J. Sci. Comput.} \bibinfo{volume}{77},
  \bibinfo{pages}{443--466}.
\bibitem[{Ainsworth and Oden(2000)}]{ainsworth_oden_2000a}
\bibinfo{author}{Ainsworth, M.}, \bibinfo{author}{Oden, J.T.},
  \bibinfo{year}{2000}.
\newblock \bibinfo{title}{A posteriori error estimation in finite element
  analysis}.
\newblock \bibinfo{publisher}{Wiley}.
\bibitem[{Araya et~al.(2019a)Araya, Solano and Vega}]{araya_solano_vega_2019a}
\bibinfo{author}{Araya, R.}, \bibinfo{author}{Solano, M.},
  \bibinfo{author}{Vega, P.}, \bibinfo{year}{2019}a.
\newblock \bibinfo{title}{Analysis of an adaptive {HDG} method for the
  {B}rinkman problem}.
\newblock \bibinfo{journal}{IMA J. Numer. Anal.} \bibinfo{volume}{39},
  \bibinfo{pages}{1502--1528}.
\bibitem[{Araya et~al.(2019b)Araya, Solano and Vega}]{araya_solano_vega_2019b}
\bibinfo{author}{Araya, R.}, \bibinfo{author}{Solano, M.},
  \bibinfo{author}{Vega, P.}, \bibinfo{year}{2019}b.
\newblock \bibinfo{title}{\textit{A posteriori} error analysis of an {HDG}
  method for the {O}seen problem}.
\newblock \bibinfo{journal}{Appl. Numer. Math.} \bibinfo{volume}{146},
  \bibinfo{pages}{291--308}.
\bibitem[{Babu{\v{s}}ka et~al.(1997)Babu{\v{s}}ka, Ihlenburg, Strouboulis and
  Gangaraj}]{babuvska1997posteriori}
\bibinfo{author}{Babu{\v{s}}ka, I.}, \bibinfo{author}{Ihlenburg, F.},
  \bibinfo{author}{Strouboulis, T.}, \bibinfo{author}{Gangaraj, S.},
  \bibinfo{year}{1997}.
\newblock \bibinfo{title}{{A posteriori error estimation for finite element
  solutions of Helmholtz'equation. {II}: {E}stimation of the pollution error}}.
\newblock \bibinfo{journal}{Internat. J. Numer. Methods Engrg.}
  \bibinfo{volume}{40}, \bibinfo{pages}{3883--3900}.
\bibitem[{Babu\v{s}ka et~al.(1997)Babu\v{s}ka, Ihlenburg, Strouboulis and
  Gangaraj}]{babuska_ihlenburg_strouboulis_gangaraj_1997}
\bibinfo{author}{Babu\v{s}ka, I.}, \bibinfo{author}{Ihlenburg, F.},
  \bibinfo{author}{Strouboulis, T.}, \bibinfo{author}{Gangaraj, S.K.},
  \bibinfo{year}{1997}.
\newblock \bibinfo{title}{A posteriori error estimation for finite element
  solutions of {H}elmholtz' equation. {I}. {T}he quality of local indicators
  and estimators}.
\newblock \bibinfo{journal}{Internat. J. Numer. Methods Engrg.}
  \bibinfo{volume}{40}, \bibinfo{pages}{3443--3462}.
\bibitem[{Babu\v{s}ka and Sauter(1997)}]{babuska1997pollution}
\bibinfo{author}{Babu\v{s}ka, I.M.}, \bibinfo{author}{Sauter, S.A.},
  \bibinfo{year}{1997}.
\newblock \bibinfo{title}{Is the pollution effect of the {FEM} avoidable for
  the {H}elmholtz equation considering high wave numbers?}
\newblock \bibinfo{journal}{SIAM J. Numer. Anal.} \bibinfo{volume}{34},
  \bibinfo{pages}{2392--2423}.
\bibitem[{Bank et~al.(1983)Bank, Sherman and Weiser}]{bank1983some}
\bibinfo{author}{Bank, R.E.}, \bibinfo{author}{Sherman, A.H.},
  \bibinfo{author}{Weiser, A.}, \bibinfo{year}{1983}.
\newblock \bibinfo{title}{Some refinement algorithms and data structures for
  regular local mesh refinement}, in: \bibinfo{editor}{Stepleman, R.S.},
  \bibinfo{editor}{Carver, M.} (Eds.), \bibinfo{booktitle}{Scientific
  Computing, Applications of Mathematics and Computing to the Physical
  Sciences}. \bibinfo{publisher}{IMACS. North-Holland}, pp.
  \bibinfo{pages}{3--17}.
\bibitem[{Bebendorf(2003)}]{Poincare_ndim}
\bibinfo{author}{Bebendorf, M.}, \bibinfo{year}{2003}.
\newblock \bibinfo{title}{A note on the {P}oincar\'e{} inequality for convex
  domains}.
\newblock \bibinfo{journal}{Z. Anal. Anwendungen} \bibinfo{volume}{22},
  \bibinfo{pages}{751--756}.
\bibitem[{Brezzi et~al.(1987)Brezzi, {J. Douglas, Jr.}, Dur\'an and
  Fortin}]{BDDF}
\bibinfo{author}{Brezzi, F.}, \bibinfo{author}{{J. Douglas, Jr.}},
  \bibinfo{author}{Dur\'an, R.}, \bibinfo{author}{Fortin, M.},
  \bibinfo{year}{1987}.
\newblock \bibinfo{title}{Mixed finite elements for second order elliptic
  problems in three variables}.
\newblock \bibinfo{journal}{Numer. Math.} \bibinfo{volume}{52},
  \bibinfo{pages}{237--250}.
\bibitem[{Brezzi et~al.(1985)Brezzi, {J. Douglas, Jr.} and Marini}]{BDM}
\bibinfo{author}{Brezzi, F.}, \bibinfo{author}{{J. Douglas, Jr.}},
  \bibinfo{author}{Marini, L.D.}, \bibinfo{year}{1985}.
\newblock \bibinfo{title}{Two families of mixed finite elements for second
  order elliptic problems}.
\newblock \bibinfo{journal}{Numer. Math.} \bibinfo{volume}{47},
  \bibinfo{pages}{217--235}.
\bibitem[{Calo et~al.(2020)Calo, Ern, Muga and Rojas}]{calo2020adaptive}
\bibinfo{author}{Calo, V.M.}, \bibinfo{author}{Ern, A.}, \bibinfo{author}{Muga,
  I.}, \bibinfo{author}{Rojas, S.}, \bibinfo{year}{2020}.
\newblock \bibinfo{title}{An adaptive stabilized conforming finite element
  method via residual minimization on dual discontinuous {G}alerkin norms}.
\newblock \bibinfo{journal}{Comput. Methods Appl. Mech. Eng.}
  \bibinfo{volume}{363}.
\newblock \bibinfo{note}{112891}.
\bibitem[{Chaumont-Frelet and Vega(2024)}]{chaumontfrelet_vega_2024}
\bibinfo{author}{Chaumont-Frelet, T.}, \bibinfo{author}{Vega, P.},
  \bibinfo{year}{2024}.
\newblock \bibinfo{title}{Frequency-explicit a posteriori error estimates for
  discontinuous {G}alerkin discretizations of {M}axwell's equations}.
\newblock \bibinfo{journal}{SIAM J. Numer. Anal.} \bibinfo{volume}{62},
  \bibinfo{pages}{400--421}.
\bibitem[{Chen et~al.(2016)Chen, Li and Qiu}]{chen_li_qiu_2016}
\bibinfo{author}{Chen, H.}, \bibinfo{author}{Li, J.}, \bibinfo{author}{Qiu,
  W.}, \bibinfo{year}{2016}.
\newblock \bibinfo{title}{Robust a posteriori error estimates for {HDG} method
  for convection--diffusion equations}.
\newblock \bibinfo{journal}{IMA J. Numer. Anal.} \bibinfo{volume}{36},
  \bibinfo{pages}{437--462}.
\bibitem[{Chen et~al.(2013)Chen, Lu and Xu}]{chen2013hybridizable}
\bibinfo{author}{Chen, H.}, \bibinfo{author}{Lu, P.}, \bibinfo{author}{Xu, X.},
  \bibinfo{year}{2013}.
\newblock \bibinfo{title}{A hybridizable discontinuous {G}alerkin method for
  the {H}elmholtz equation with high wave number}.
\newblock \bibinfo{journal}{SIAM J. Numer. Anal.} \bibinfo{volume}{51},
  \bibinfo{pages}{2166--2188}.
\bibitem[{Chen et~al.(2018)Chen, Qiu and Shi}]{chen_qiu_shi_2018}
\bibinfo{author}{Chen, H.}, \bibinfo{author}{Qiu, W.}, \bibinfo{author}{Shi,
  K.}, \bibinfo{year}{2018}.
\newblock \bibinfo{title}{A priori and computable a posteriori error estimates
  for an {HDG} method for the coercive {M}axwell equations}.
\newblock \bibinfo{journal}{Comput. Methods Appl. Mech. Eng.}
  \bibinfo{volume}{333}, \bibinfo{pages}{287--310}.
\bibitem[{Ciarlet(2002)}]{ciarlet}
\bibinfo{author}{Ciarlet, P.G.}, \bibinfo{year}{2002}.
\newblock \bibinfo{title}{The finite element method for elliptic problems}.
  volume~\bibinfo{volume}{40} of \textit{\bibinfo{series}{Classics in Applied
  Mathematics}}.
\newblock \bibinfo{publisher}{Society for Industrial and Applied Mathematics
  (SIAM), Philadelphia, PA}.
\bibitem[{Cier et~al.(2021)Cier, Rojas and Calo}]{cier2021automatically}
\bibinfo{author}{Cier, R.J.}, \bibinfo{author}{Rojas, S.},
  \bibinfo{author}{Calo, V.M.}, \bibinfo{year}{2021}.
\newblock \bibinfo{title}{Automatically adaptive, stabilized finite element
  method via residual minimization for heterogeneous, anisotropic
  advection--diffusion--reaction problems}.
\newblock \bibinfo{journal}{Comput. Methods Appl. Mech. Engrg.}
  \bibinfo{volume}{385}, \bibinfo{pages}{114027}.
\bibitem[{Cockburn et~al.(2008)Cockburn, Dong and Gumz\'an}]{CoDoGu2008}
\bibinfo{author}{Cockburn, B.}, \bibinfo{author}{Dong, B.},
  \bibinfo{author}{Gumz\'an, J.}, \bibinfo{year}{2008}.
\newblock \bibinfo{title}{A superconvergent {LDG}--hybridizable galerkin method
  for second-order elliptic problems}.
\newblock \bibinfo{journal}{Math. Comp.} \bibinfo{volume}{77},
  \bibinfo{pages}{1887--1916}.
\bibitem[{Cockburn et~al.(2009)Cockburn, Gopalakrishnan and
  Lazarov}]{CoGoLa2009}
\bibinfo{author}{Cockburn, B.}, \bibinfo{author}{Gopalakrishnan, J.},
  \bibinfo{author}{Lazarov, R.D.}, \bibinfo{year}{2009}.
\newblock \bibinfo{title}{Unified hybridization of discontinuous {G}alerkin,
  mixed, and continuous {G}alerkin methods for second order elliptic problems}.
\newblock \bibinfo{journal}{SIAM J. Numer. Anal.} \bibinfo{volume}{47},
  \bibinfo{pages}{1319--1365}.
\bibitem[{Cockburn and Zhang(2012)}]{MR2914423}
\bibinfo{author}{Cockburn, B.}, \bibinfo{author}{Zhang, W.},
  \bibinfo{year}{2012}.
\newblock \bibinfo{title}{A posteriori error estimates for {HDG} methods}.
\newblock \bibinfo{journal}{J. Sci. Comput.} \bibinfo{volume}{51},
  \bibinfo{pages}{582--607}.
\bibitem[{Cockburn and Zhang(2013)}]{MR3033028}
\bibinfo{author}{Cockburn, B.}, \bibinfo{author}{Zhang, W.},
  \bibinfo{year}{2013}.
\newblock \bibinfo{title}{A posteriori error analysis for hybridizable
  discontinuous {G}alerkin methods for second order elliptic problems}.
\newblock \bibinfo{journal}{SIAM J. Numer. Anal.} \bibinfo{volume}{51},
  \bibinfo{pages}{676--693}.
\bibitem[{Cockburn and Zhang(2014)}]{MR3167450}
\bibinfo{author}{Cockburn, B.}, \bibinfo{author}{Zhang, W.},
  \bibinfo{year}{2014}.
\newblock \bibinfo{title}{An a posteriori error estimate for the
  variable-degree {R}aviart-{T}homas method}.
\newblock \bibinfo{journal}{Math. Comp.} \bibinfo{volume}{83},
  \bibinfo{pages}{1063--1082}.
\bibitem[{Cohen et~al.(2012)Cohen, Dahmen and Welper}]{CohDahWelM2AN2012}
\bibinfo{author}{Cohen, A.}, \bibinfo{author}{Dahmen, W.},
  \bibinfo{author}{Welper, G.}, \bibinfo{year}{2012}.
\newblock \bibinfo{title}{Adaptivity and variational stabilization for
  convection-diffusion equations}.
\newblock \bibinfo{journal}{ESAIM Math. Model. Numer. Anal.}
  \bibinfo{volume}{46}, \bibinfo{pages}{1247--1273}.
\bibitem[{Cui and Zhang(2013)}]{cui2013analysis}
\bibinfo{author}{Cui, J.}, \bibinfo{author}{Zhang, W.}, \bibinfo{year}{2013}.
\newblock \bibinfo{title}{An analysis of {HDG} methods for the {H}elmholtz
  equation}.
\newblock \bibinfo{journal}{IMA J. Numer. Anal.} \bibinfo{volume}{34},
  \bibinfo{pages}{279--295}.
\bibitem[{Darrigrand et~al.(2015)Darrigrand, Pardo and
  Muga}]{darrigrand_pardo_muga_2015}
\bibinfo{author}{Darrigrand, V.}, \bibinfo{author}{Pardo, D.},
  \bibinfo{author}{Muga, I.}, \bibinfo{year}{2015}.
\newblock \bibinfo{title}{Goal-oriented adaptivity using unconventional error
  representations for the 1{D} {H}elmholtz equation}.
\newblock \bibinfo{journal}{Comput. Math. Appl.} \bibinfo{volume}{69},
  \bibinfo{pages}{964--979}.
\bibitem[{Demkowicz(2006)}]{demkowicz_2006a}
\bibinfo{author}{Demkowicz, L.}, \bibinfo{year}{2006}.
\newblock \bibinfo{title}{Computing with $hp$-adaptive finite elements}.
  volume~\bibinfo{volume}{1}.
\newblock \bibinfo{publisher}{Wiley}.
\bibitem[{Demkowicz and Gopalakrishnan(2010)}]{DemGopCMAME2010}
\bibinfo{author}{Demkowicz, L.}, \bibinfo{author}{Gopalakrishnan, J.},
  \bibinfo{year}{2010}.
\newblock \bibinfo{title}{A class of discontinuous {P}etrov--{G}alerkin
  methods. {P}art {I}. {T}he transport equation}.
\newblock \bibinfo{journal}{Comput. Methods Appl. Mech. Engrg.}
  \bibinfo{volume}{199}, \bibinfo{pages}{1558--1572}.
\bibitem[{Demkowicz and Gopalakrishnan(2014)}]{DemGopBOOK-CH2014}
\bibinfo{author}{Demkowicz, L.}, \bibinfo{author}{Gopalakrishnan, J.},
  \bibinfo{year}{2014}.
\newblock \bibinfo{title}{An overview of the discontinuous {Petrov {G}alerkin}
  method}, in: \bibinfo{editor}{Feng, X.}, \bibinfo{editor}{Karakashian, O.},
  \bibinfo{editor}{Xing, Y.} (Eds.), \bibinfo{booktitle}{Recent Developments in
  Discontinuous {G}alerkin Finite Element Methods for Partial Differential
  Equations: 2012 John H Barrett Memorial Lectures}.
  \bibinfo{publisher}{Springer}, \bibinfo{address}{Cham}. volume
  \bibinfo{volume}{157} of \textit{\bibinfo{series}{The IMA Volumes in
  Mathematics and its Applications}}, pp. \bibinfo{pages}{149--180}.
\bibitem[{Demkowicz et~al.(2012)Demkowicz, Gopalakrishnan and Niemi}]{DPG3}
\bibinfo{author}{Demkowicz, L.}, \bibinfo{author}{Gopalakrishnan, J.},
  \bibinfo{author}{Niemi, A.H.}, \bibinfo{year}{2012}.
\newblock \bibinfo{title}{A class of discontinuous {P}etrov-{G}alerkin methods.
  {P}art {III}: {A}daptivity}.
\newblock \bibinfo{journal}{Appl. Numer. Math.} \bibinfo{volume}{62},
  \bibinfo{pages}{396--427}.
\bibitem[{D{\"o}rfler(1996)}]{dorfler1996convergent}
\bibinfo{author}{D{\"o}rfler, W.}, \bibinfo{year}{1996}.
\newblock \bibinfo{title}{A convergent adaptive algorithm for {P}oisson's
  equation}.
\newblock \bibinfo{journal}{SIAM J. Numer. Anal.} \bibinfo{volume}{33},
  \bibinfo{pages}{1106--1124}.
\bibitem[{D\"orfler and Sauter(2013)}]{dorfler_sauter_2013a}
\bibinfo{author}{D\"orfler, W.}, \bibinfo{author}{Sauter, S.},
  \bibinfo{year}{2013}.
\newblock \bibinfo{title}{A posteriori error estimation for highly indefinite
  {H}elmholtz problems}.
\newblock \bibinfo{journal}{Comput. Meth. Appl. Math.} \bibinfo{volume}{13},
  \bibinfo{pages}{333--347}.
\bibitem[{Du and Sayas(2019)}]{MR3970243}
\bibinfo{author}{Du, S.}, \bibinfo{author}{Sayas, F.J.}, \bibinfo{year}{2019}.
\newblock \bibinfo{title}{An invitation to the theory of the hybridizable
  discontinuous {G}alerkin method. Projections, estimates, tools}.
\newblock SpringerBriefs in Mathematics, \bibinfo{publisher}{Springer, Cham}.
\bibitem[{Ern and Guermond(2021a)}]{ern_guermond_2021a}
\bibinfo{author}{Ern, A.}, \bibinfo{author}{Guermond, J.L.},
  \bibinfo{year}{2021}a.
\newblock \bibinfo{title}{Finite elements {I}. {A}pproximation and
  {I}nterpolation}.
\newblock \bibinfo{publisher}{Springer, Cham}.
\bibitem[{Ern and Guermond(2021b)}]{ern_guermond_2021b}
\bibinfo{author}{Ern, A.}, \bibinfo{author}{Guermond, J.L.},
  \bibinfo{year}{2021}b.
\newblock \bibinfo{title}{Finite elements {II}. {G}alerkin approximation,
  elliptic and mixed {PDE}s}.
\newblock \bibinfo{publisher}{Springer, Cham}.
\bibitem[{Farhat et~al.(2003)Farhat, Harari and
  Hetmaniuk}]{farhat_harari_hetmaniuk_2003}
\bibinfo{author}{Farhat, C.}, \bibinfo{author}{Harari, I.},
  \bibinfo{author}{Hetmaniuk, U.}, \bibinfo{year}{2003}.
\newblock \bibinfo{title}{A discontinuous {G}alerkin method with {L}agrange
  multipliers for the solution of {H}elmholtz problems in the mid-frequency
  regime}.
\newblock \bibinfo{journal}{Comput. Methods Appl. Mech. Engrg.}
  \bibinfo{volume}{192}, \bibinfo{pages}{1389--1419}.
\bibitem[{Feng and Wu(2009)}]{feng2009discontinuous}
\bibinfo{author}{Feng, X.}, \bibinfo{author}{Wu, H.}, \bibinfo{year}{2009}.
\newblock \bibinfo{title}{Discontinuous {G}alerkin methods for the {H}elmholtz
  equation with large wave number}.
\newblock \bibinfo{journal}{SIAM J. Numer. Anal.} \bibinfo{volume}{47},
  \bibinfo{pages}{2872--2896}.
\bibitem[{Feng and Wu(2011)}]{feng_wu_2011}
\bibinfo{author}{Feng, X.}, \bibinfo{author}{Wu, H.}, \bibinfo{year}{2011}.
\newblock \bibinfo{title}{{$hp$}-discontinuous {G}alerkin methods for the
  {H}elmholtz equation with large wave number}.
\newblock \bibinfo{journal}{Math. Comp.} \bibinfo{volume}{80},
  \bibinfo{pages}{1997--2024}.
\bibitem[{Feng and Xing(2013)}]{feng_xing_2013}
\bibinfo{author}{Feng, X.}, \bibinfo{author}{Xing, Y.}, \bibinfo{year}{2013}.
\newblock \bibinfo{title}{Absolutely stable local discontinuous {G}alerkin
  methods for the {H}elmholtz equation with large wave number}.
\newblock \bibinfo{journal}{Math. Comp.} \bibinfo{volume}{82},
  \bibinfo{pages}{1269--1296}.
\bibitem[{Führer and Sánchez(2024)}]{tofu_manuel_2024}
\bibinfo{author}{Führer, T.}, \bibinfo{author}{Sánchez, M.A.},
  \bibinfo{year}{2024}.
\newblock \bibinfo{title}{Quasi-interpolators with application to
  postprocessing in finite element methods}.
\newblock \bibinfo{note}{Submitted, preprint arXiv:2404.13183 [math.NA]}.
\bibitem[{Griesmaier and Monk(2011)}]{griesmaier2011error}
\bibinfo{author}{Griesmaier, R.}, \bibinfo{author}{Monk, P.},
  \bibinfo{year}{2011}.
\newblock \bibinfo{title}{Error analysis for a hybridizable discontinuous
  {G}alerkin method for the {H}elmholtz equation}.
\newblock \bibinfo{journal}{J. Sci. Comput.} \bibinfo{volume}{49},
  \bibinfo{pages}{291--310}.
\bibitem[{Harari and Hughes(1991)}]{harari_hughes_1991}
\bibinfo{author}{Harari, I.}, \bibinfo{author}{Hughes, T.J.R.},
  \bibinfo{year}{1991}.
\newblock \bibinfo{title}{Finite element methods for the {H}elmholtz equation
  in an exterior domain: model problems}.
\newblock \bibinfo{journal}{Comput. Methods Appl. Mech. Engrg.}
  \bibinfo{volume}{87}, \bibinfo{pages}{59--96}.
\bibitem[{Harari and Hughes(1992)}]{harari_hughes_1992}
\bibinfo{author}{Harari, I.}, \bibinfo{author}{Hughes, T.J.R.},
  \bibinfo{year}{1992}.
\newblock \bibinfo{title}{Galerkin/least-squares finite element methods for the
  reduced wave equation with nonreflecting boundary conditions in unbounded
  domains}.
\newblock \bibinfo{journal}{Comput. Methods Appl. Mech. Engrg.}
  \bibinfo{volume}{98}, \bibinfo{pages}{411--454}.
\bibitem[{Hiptmair et~al.(2016)Hiptmair, Moiola and Perugia}]{Hiptmair2016}
\bibinfo{author}{Hiptmair, R.}, \bibinfo{author}{Moiola, A.},
  \bibinfo{author}{Perugia, I.}, \bibinfo{year}{2016}.
\newblock \bibinfo{title}{A survey of {T}refftz methods for the {H}elmholtz
  equation}, in: \bibinfo{editor}{Barrenechea, G.R.}, \bibinfo{editor}{Brezzi,
  F.}, \bibinfo{editor}{Cangiani, A.}, \bibinfo{editor}{Georgoulis, E.H.}
  (Eds.), \bibinfo{booktitle}{Building Bridges: Connections and Challenges in
  Modern Approaches to Numerical Partial Differential Equations}.
  \bibinfo{publisher}{Springer International Publishing},
  \bibinfo{address}{Cham}, pp. \bibinfo{pages}{237--279}.
\bibitem[{Hu and Song(2020)}]{hu_song_2020}
\bibinfo{author}{Hu, Q.}, \bibinfo{author}{Song, R.}, \bibinfo{year}{2020}.
\newblock \bibinfo{title}{A novel least squares method for {H}elmholtz
  equations with large wave numbers}.
\newblock \bibinfo{journal}{SIAM J. Numer. Anal.} \bibinfo{volume}{58},
  \bibinfo{pages}{3091--3123}.
\bibitem[{Ihlenburg and Babu\v{s}ka(1995)}]{ihlenburg_babuska_1995}
\bibinfo{author}{Ihlenburg, F.}, \bibinfo{author}{Babu\v{s}ka, I.},
  \bibinfo{year}{1995}.
\newblock \bibinfo{title}{Finite element solution of the {H}elmholtz equation
  with high wave number. {I}. {T}he {$h$}-version of the {FEM}}.
\newblock \bibinfo{journal}{Comput. Math. Appl.} \bibinfo{volume}{30},
  \bibinfo{pages}{9--37}.
\bibitem[{Ihlenburg and Babu\v{s}ka(1997)}]{ihlenburg_babuska_1997}
\bibinfo{author}{Ihlenburg, F.}, \bibinfo{author}{Babu\v{s}ka, I.},
  \bibinfo{year}{1997}.
\newblock \bibinfo{title}{Finite element solution of the {H}elmholtz equation
  with high wave number. {II}. {T}he {$h$}-{$p$} version of the {FEM}}.
\newblock \bibinfo{journal}{SIAM J. Numer. Anal.} \bibinfo{volume}{34},
  \bibinfo{pages}{315--358}.
\bibitem[{Kyburg et~al.(2022)Kyburg, Rojas and Calo}]{kyburg2022incompressible}
\bibinfo{author}{Kyburg, F.E.}, \bibinfo{author}{Rojas, S.},
  \bibinfo{author}{Calo, V.M.}, \bibinfo{year}{2022}.
\newblock \bibinfo{title}{Incompressible flow modeling using an adaptive
  stabilized finite element method based on residual minimization}.
\newblock \bibinfo{journal}{Int. J. Numer. Methods Eng.} \bibinfo{volume}{123},
  \bibinfo{pages}{1717--1735}.
\bibitem[{Leng(2021)}]{leng_2021}
\bibinfo{author}{Leng, H.}, \bibinfo{year}{2021}.
\newblock \bibinfo{title}{Adaptive {HDG} methods for the steady-state
  incompressible {N}avier–{S}tokes equation}.
\newblock \bibinfo{journal}{J. Sci. Comp.} \bibinfo{volume}{87},
  \bibinfo{pages}{37}.
\bibitem[{Li et~al.(2023)Li, Liu and Yang}]{li_liu_yang_2023}
\bibinfo{author}{Li, R.}, \bibinfo{author}{Liu, Q.}, \bibinfo{author}{Yang,
  F.}, \bibinfo{year}{2023}.
\newblock \bibinfo{title}{A discontinuous least squares finite element method
  for the {H}elmholtz equation}.
\newblock \bibinfo{journal}{Numer. Methods Partial Differ. Eq.}
  \bibinfo{volume}{39}, \bibinfo{pages}{1425--1448}.
\bibitem[{{\L}o{\'s} et~al.(2021){\L}o{\'s}, Rojas, Paszy{\'n}ski, Muga and
  Calo}]{los2021dgirm}
\bibinfo{author}{{\L}o{\'s}, M.}, \bibinfo{author}{Rojas, S.},
  \bibinfo{author}{Paszy{\'n}ski, M.}, \bibinfo{author}{Muga, I.},
  \bibinfo{author}{Calo, V.M.}, \bibinfo{year}{2021}.
\newblock \bibinfo{title}{{DGIRM}: {D}iscontinuous {G}alerkin based
  isogeometric residual minimization for the {S}tokes problem}.
\newblock \bibinfo{journal}{J. Comput. Sci.} \bibinfo{volume}{50}.
\newblock \bibinfo{note}{101306}.
\bibitem[{Melenk et~al.(2013)Melenk, Parsania and
  Sauter}]{melenk_parsiania_sauter_2013}
\bibinfo{author}{Melenk, J.M.}, \bibinfo{author}{Parsania, A.},
  \bibinfo{author}{Sauter, S.}, \bibinfo{year}{2013}.
\newblock \bibinfo{title}{General {DG}-methods for highly indefinite
  {H}elmholtz problems}.
\newblock \bibinfo{journal}{J. Sci. Comput.} \bibinfo{volume}{57},
  \bibinfo{pages}{536--581}.
\bibitem[{Melenk and Sauter(2011)}]{melenk_sauter_2011}
\bibinfo{author}{Melenk, J.M.}, \bibinfo{author}{Sauter, S.},
  \bibinfo{year}{2011}.
\newblock \bibinfo{title}{Wavenumber explicit convergence analysis for
  {G}alerkin discretizations of the {H}elmholtz equation}.
\newblock \bibinfo{journal}{SIAM J. Numer. Anal.} \bibinfo{volume}{49},
  \bibinfo{pages}{1210--1243}.
\bibitem[{Monsuur and Stevenson(2023)}]{monsuur2023pollution}
\bibinfo{author}{Monsuur, H.}, \bibinfo{author}{Stevenson, R.},
  \bibinfo{year}{2023}.
\newblock \bibinfo{title}{A pollution-free ultra-weak {FOSLS} discretization of
  the {H}elmholtz equation}.
\newblock \bibinfo{journal}{Comput. Math. Appl.} \bibinfo{volume}{148},
  \bibinfo{pages}{241--255}.
\bibitem[{Muga et~al.(2023)Muga, Rojas and Vega}]{muga_rojas_vega_2022a}
\bibinfo{author}{Muga, I.}, \bibinfo{author}{Rojas, S.}, \bibinfo{author}{Vega,
  P.}, \bibinfo{year}{2023}.
\newblock \bibinfo{title}{An adaptive superconvergent mixed finite element
  method based on local residual minimization}.
\newblock \bibinfo{journal}{SIAM J. Numer. Anal.} \bibinfo{volume}{61},
  \bibinfo{pages}{2084--2105}.
\bibitem[{Nguyen et~al.(2015)Nguyen, Peraire, Reitich and
  Cockburn}]{nguyen_peraire_reitich_cockburn_2015}
\bibinfo{author}{Nguyen, N.C.}, \bibinfo{author}{Peraire, J.},
  \bibinfo{author}{Reitich, F.}, \bibinfo{author}{Cockburn, B.},
  \bibinfo{year}{2015}.
\newblock \bibinfo{title}{A phase-based hybridizable discontinuous {G}alerkin
  method for the numerical solution of the {H}elmholtz equation}.
\newblock \bibinfo{journal}{J. Comput. Phys.} \bibinfo{volume}{290},
  \bibinfo{pages}{318--335}.
\bibitem[{Payne and Weinberger(1960)}]{PaWe1960}
\bibinfo{author}{Payne, L.E.}, \bibinfo{author}{Weinberger, H.F.},
  \bibinfo{year}{1960}.
\newblock \bibinfo{title}{An optimal {P}oincar\'e inequality for convex
  domains}.
\newblock \bibinfo{journal}{Arch. Rat. Mech. Anal.} \bibinfo{volume}{25},
  \bibinfo{pages}{286--292}.
\bibitem[{Peraire and Patera(1999)}]{peraire_patera_1999}
\bibinfo{author}{Peraire, J.}, \bibinfo{author}{Patera, A.T.},
  \bibinfo{year}{1999}.
\newblock \bibinfo{title}{Asymptotic a posteriori finite element bounds for the
  outputs of noncoercive problems: the {H}elmholtz and {B}urgers equations}.
\newblock \bibinfo{journal}{Comput. Methods Appl. Mech. Engrg.}
  \bibinfo{volume}{171}, \bibinfo{pages}{77--86}.
\bibitem[{Sauter and Zech(2015)}]{sauter_zech_2015a}
\bibinfo{author}{Sauter, S.}, \bibinfo{author}{Zech, J.}, \bibinfo{year}{2015}.
\newblock \bibinfo{title}{A posteriori error estimation of $hp$-d{G} finite
  element methods for highly indefinite {H}elmholtz problems}.
\newblock \bibinfo{journal}{SIAM J. Numer. Anal.} \bibinfo{volume}{53},
  \bibinfo{pages}{2414--2440}.
\bibitem[{Sch{\"o}berl(1997)}]{schoberl1997netgen}
\bibinfo{author}{Sch{\"o}berl, J.}, \bibinfo{year}{1997}.
\newblock \bibinfo{title}{Netgen an advancing front 2d/3d-mesh generator based
  on abstract rules}.
\newblock \bibinfo{journal}{Computing and visualization in science}
  \bibinfo{volume}{1}, \bibinfo{pages}{41--52}.
\bibitem[{Sch{\"o}berl(2014)}]{schoberl2014c++}
\bibinfo{author}{Sch{\"o}berl, J.}, \bibinfo{year}{2014}.
\newblock \bibinfo{title}{C++ 11 implementation of finite elements in
  {NGS}olve}.
\newblock \bibinfo{journal}{Institute for analysis and scientific computing,
  Vienna University of Technology} \bibinfo{volume}{30}.
\bibitem[{Stenberg(1991)}]{stenberg_1991}
\bibinfo{author}{Stenberg, R.}, \bibinfo{year}{1991}.
\newblock \bibinfo{title}{Postprocessing schemes for some mixed finite
  elements}.
\newblock \bibinfo{journal}{RAIRO Modél. Math. Anal. Numér.}
  \bibinfo{volume}{25}, \bibinfo{pages}{151--167}.
\bibitem[{Stewart and Hughes(1996a)}]{stewart_hughes_1996a}
\bibinfo{author}{Stewart, J.R.}, \bibinfo{author}{Hughes, T.J.R.},
  \bibinfo{year}{1996}a.
\newblock \bibinfo{title}{Explicit residual-based a posteriori error estimation
  for finite element discretizations of the {H}elmholtz equation: computation
  of the constant and new measures of error estimator quality}.
\newblock \bibinfo{journal}{Comput. Methods Appl. Mech. Engrg.}
  \bibinfo{volume}{131}, \bibinfo{pages}{335--363}.
\bibitem[{Stewart and Hughes(1996b)}]{stewart_hughes_1996b}
\bibinfo{author}{Stewart, J.R.}, \bibinfo{author}{Hughes, T.J.R.},
  \bibinfo{year}{1996}b.
\newblock \bibinfo{title}{A posteriori error estimation and adaptive finite
  element computation of the {H}elmholtz equation in exterior domains}.
\newblock \bibinfo{journal}{Finite Elem. Anal. Des.} \bibinfo{volume}{22},
  \bibinfo{pages}{15--24}.
\bibitem[{Stewart and Hughes(1997a)}]{stewart_hughes_1997}
\bibinfo{author}{Stewart, J.R.}, \bibinfo{author}{Hughes, T.J.R.},
  \bibinfo{year}{1997}a.
\newblock \bibinfo{title}{An a posteriori error estimator and {$hp$}-adaptive
  strategy for finite element discretizations of the {H}elmholtz equation in
  exterior domains}.
\newblock \bibinfo{journal}{Finite Elem. Anal. Des.} \bibinfo{volume}{25},
  \bibinfo{pages}{1--26}.
\bibitem[{Stewart and Hughes(1997b)}]{stewart_hughes_1997b}
\bibinfo{author}{Stewart, J.R.}, \bibinfo{author}{Hughes, T.J.R.},
  \bibinfo{year}{1997}b.
\newblock \bibinfo{title}{{$h$}-adaptive finite element computation of
  time-harmonic exterior acoustics problems in two dimensions}.
\newblock \bibinfo{journal}{Comput. Methods Appl. Mech. Engrg.}
  \bibinfo{volume}{146}, \bibinfo{pages}{65--89}.
\bibitem[{Thompson and Pinsky(1995)}]{thompson_pinsky_1995}
\bibinfo{author}{Thompson, L.L.}, \bibinfo{author}{Pinsky, P.M.},
  \bibinfo{year}{1995}.
\newblock \bibinfo{title}{A {G}alerkin least-squares finite element method for
  the two-dimensional {H}elmholtz equation}.
\newblock \bibinfo{journal}{Internat. J. Numer. Methods Engrg.}
  \bibinfo{volume}{38}, \bibinfo{pages}{371--397}.
\bibitem[{Verf{\"u}hrt(1996)}]{verfuhrt1996review}
\bibinfo{author}{Verf{\"u}hrt, R.}, \bibinfo{year}{1996}.
\newblock \bibinfo{title}{A review of a posteriori error estimation and
  adaptive mesh-refinement rechniques}.
\newblock Advances in Numerical Mathematics,
  \bibinfo{publisher}{Wiley-Teubner}.
\bibitem[{Verf{\"u}rth(1998)}]{verfurth}
\bibinfo{author}{Verf{\"u}rth, R.}, \bibinfo{year}{1998}.
\newblock \bibinfo{title}{A posteriori error estimators for
  convection-diffusion equations}.
\newblock \bibinfo{journal}{Numer. Math.} \bibinfo{volume}{80},
  \bibinfo{pages}{641--663}.
\bibitem[{Verf\"{u}rth(2013)}]{verfurth_2013}
\bibinfo{author}{Verf\"{u}rth, R.}, \bibinfo{year}{2013}.
\newblock \bibinfo{title}{A posteriori error estimation techniques for finite
  element methods}.
\newblock Numerical Mathematics and Scientific Computation,
  \bibinfo{publisher}{Oxford University Press, Oxford}.
\bibitem[{Wu(2014)}]{zhu_wu_2014}
\bibinfo{author}{Wu, H.}, \bibinfo{year}{2014}.
\newblock \bibinfo{title}{Pre-asymptotic error analysis of {CIP}-{FEM} and
  {FEM} for the {H}elmholtz equation with high wave number. {P}art {I}: linear
  version}.
\newblock \bibinfo{journal}{IMA J. Numer. Anal.} \bibinfo{volume}{34},
  \bibinfo{pages}{1266--1288}.
\bibitem[{Zhu and Wu(2021)}]{zhu2021preasymptotic}
\bibinfo{author}{Zhu, B.}, \bibinfo{author}{Wu, H.}, \bibinfo{year}{2021}.
\newblock \bibinfo{title}{Preasymptotic error analysis of the {HDG} method for
  {H}elmholtz equation with large wave number}.
\newblock \bibinfo{journal}{J. Sci. Comput.} \bibinfo{volume}{87},
  \bibinfo{pages}{1--34}.
\bibitem[{Zhu and Wu(2024)}]{zhu_wu_2024}
\bibinfo{author}{Zhu, B.}, \bibinfo{author}{Wu, H.}, \bibinfo{year}{2024}.
\newblock \bibinfo{title}{The {$(p,p-1)$}-{HDG} method for the {H}elmholtz
  equation with large wave number}.
\newblock \bibinfo{journal}{SIAM J. Numer. Anal.} \bibinfo{volume}{62},
  \bibinfo{pages}{1394--1419}.
\bibitem[{Zhu and Wu(2013)}]{zhu_wu_2013}
\bibinfo{author}{Zhu, L.}, \bibinfo{author}{Wu, H.}, \bibinfo{year}{2013}.
\newblock \bibinfo{title}{Preasymptotic error analysis of {CIP}-{FEM} and {FEM}
  for {H}elmholtz equation with high wave number. {P}art {II}: {$hp$} version}.
\newblock \bibinfo{journal}{SIAM J. Numer. Anal.} \bibinfo{volume}{51},
  \bibinfo{pages}{1828--1852}.
\bibitem[{Zienkiewicz and Zhu(1987)}]{zz_1987}
\bibinfo{author}{Zienkiewicz, O.C.}, \bibinfo{author}{Zhu, J.Z.},
  \bibinfo{year}{1987}.
\newblock \bibinfo{title}{A simple error estimator and adaptive procedure for
  practical engineering analysis}.
\newblock \bibinfo{journal}{Internat. J. Numer. Methods Engrg.}
  \bibinfo{volume}{24}, \bibinfo{pages}{337--357}.

\end{thebibliography}






\end{document}